\documentclass[11pt]{article}
\pdfoutput=1 
\usepackage{amssymb}
\usepackage[margin=1in]{geometry}

\usepackage{lineno}
\usepackage{graphicx}
\usepackage{moreverb,url}
\usepackage[ruled,vlined]{algorithm2e}
\usepackage[colorlinks,bookmarksopen,bookmarksnumbered,citecolor=red,urlcolor=red]{hyperref}

\usepackage{xcolor,tabularx,nccmath}
\usepackage{amsmath,amsthm,enumitem}
\usepackage{booktabs,caption,subcaption,mathtools,multirow}

\newcommand{\bigM}{\mathbb{M}}
\DeclareMathAlphabet{\mathdutchcal}{U}{dutchcal}{m}{n}
\SetMathAlphabet{\mathdutchcal}{bold}{U}{dutchcal}{b}{n}
\DeclareMathAlphabet{\mathdutchbcal}{U}{dutchcal}{b}{n}
\newcommand{\sA}{\mathdutchcal{A}}
\newcommand{\sB}{\mathdutchcal{B}}

\newcommand{\sV}{\mathdutchcal{V}}
\newcommand{\sT}{\mathdutchcal{T}}

\newcommand{\sE}{\mathdutchcal{E}}

\newcommand{\sL}{\mathdutchcal{L}}

\newcommand{\sP}{\mathdutchbcal{P}}

\newcommand{\sR}{\mathdutchcal{R}}
\newcommand{\sC}{\mathdutchcal{C}}

\usepackage[nameinlink,capitalise]{cleveref}
\hypersetup{colorlinks={true},allcolors={blue}}

\crefformat{section}{\S#2#1#3}
\crefformat{subsection}{\S#2#1#3}
\crefformat{equation}{#2Equation~\textup{(#1)}#3}
\crefformat{cons}{#2Constraint~\textup{(#1)}#3}
\labelcrefformat{cons}{\textup{#2(#1)#3}}
\crefformat{consset}{#2Constraints~\textup{(#1)}#3} 

\labelcrefformat{consset}{\textup{#2(#1)#3}}
\crefformat{func}{#2Function~\textup{(#1)}#3}
\labelcrefformat{func}{#2Function~\textup{(#1)}#3}
\crefformat{algo}{#2Algorithm~\textup{#1}#3}
\crefformat{objfunc}{#2Objective function~\textup{(#1)}#3}
\labelcrefformat{objfunc}{\textup{#2(#1)#3}}
\crefformat{ch}{#2Chapter~\textup{#1}#3}
\crefformat{sec}{#2Section~\textup{#1}#3}
\crefformat{tab}{#2Table~\textup{#1}#3}
\crefformat{fig}{#2Figure~\textup{#1}#3}
\crefdefaultlabelformat{#2#1#3} 
\Crefname{figure}{Figure}{Figures}
\crefformat{eq}{#2equation~\textup{(#1)}#3}
\labelcrefformat{eq}{\textup{(#2#1#3)}}
\crefformat{theo}{#2Theorem~\textup{#1}#3}
\crefformat{lem}{#2Lemma~\textup{#1}#3}
\crefformat{def}{#2Definition~\textup{#1}#3}
\crefformat{app}{#2\textup{\!#1}#3}
\crefformat{scenario}{#2Scenario~\textup{#1}#3}
 \AtBeginDocument{%
    \crefname{figure}{Figure}{figures}%
}
\let\origref\cref
\def\cref#1{\origref{#1}}
\renewenvironment{proof}[1][\proofname]{{\bfseries #1.}}{\qed}

\expandafter\def\expandafter\UrlBreaks\expandafter{\UrlBreaks
  \do\a\do\b\do\c\do\d\do\e\do\f\do\g\do\h\do\i\do\j%
  \do\k\do\l\do\m\do\n\do\o\do\p\do\q\do\r\do\s\do\t%
  \do\u\do\v\do\w\do\x\do\y\do\z\do\A\do\B\do\C\do\D%
  \do\E\do\F\do\G\do\H\do\I\do\J\do\K\do\L\do\M\do\N%
  \do\O\do\P\do\Q\do\R\do\S\do\T\do\U\do\V\do\W\do\X%
  \do\Y\do\Z}

\usepackage{stackengine}

\usepackage{algpseudocode}

\DeclarePairedDelimiter\ceil{\lceil}{\rceil}

\usepackage[normalem]{ulem} 
 
\newcommand{\norm}[1]{\lvert #1 \rvert}
\newcommand\omicron{o}

\newtheorem{theorem}{Theorem}
\newtheorem{lemma}[theorem]{Lemma}
\crefname{defi}{definition}{definitions}
\Crefname{defi}{Definition}{Definitions}
\crefname{lemma}{lemma}{lemmas}
\Crefname{lemma}{Lemma}{Lemmas}
\crefname{assumption}{assumption}{assumptions}
\Crefname{assumption}{Assumption}{Assumptions}

\usepackage{multicol}

\usepackage{authblk}
\providecommand{\keywords}[1]{\noindent\textbf{Keywords:} #1}

\usepackage{natbib}

\begin{document}

\title{Heuristic Solutions to the Single Depot Electric Vehicle Scheduling Problem with Next Day Operability Constraints}


\author[1]{Amir Davatgari}

\affil[1]{University of Illinois Chicago, 1200 W. Harrison St., Chicago, 60607, IL, USA}

\author[2,3]{Taner Cokyasar}
\author[2]{Omer Verbas}
\author[1]{Abolfazl (Kouros) Mohammadian}

\affil[2]{Argonne National Laboratory, 9700 S. Cass Avenue, Lemont, 60439, IL, USA}
\affil[3]{TrOpt R\&D, Balcali mah., Saricam, 01330, Adana, Turkey}

\maketitle


\begin{abstract}
This study focuses on the single depot electric vehicle scheduling problem (SDEVSP) within the broader context of the vehicle scheduling problem (VSP). By developing an effective scheduling model using mixed-integer linear programming (MILP), we generate bus blocks that accommodate EVs, ensuring successful completion of each block while considering recharging requirements between blocks and during off-hours. Next day operability constraints are also incorporated, allowing for seamless repetition of blocks on subsequent days. The SDEVSP is known to be computationally complex, deriving optimal solutions unattainable for large-scale problems within reasonable timeframes. To address this, we propose a two-step solution approach: first solving the single depot VSP (SDVSP), and then addressing the block chaining problem (BCP) using the blocks generated in the first step. The BCP focuses on optimizing block combinations to facilitate recharging between consecutive blocks, considering operational constraints. A case study conducted reveals that nearly 100\% electrification for Chicago, IL and Austin, TX transit buses is viable yet requires 1.6 EVs at 150-mile range per diesel vehicle.
\end{abstract}


\section{Introduction}\label[sec]{intro}

Public transportation plays a crucial role in cities by providing accessible, affordable, efficient, and equitable mobility options for travelers while helping to alleviate congestion. However, the use of conventional diesel vehicles (DVs) contributes to air pollution and carbon emissions, influencing air quality and public health \citep{FTA}. Electrification of transit buses has emerged as a solution to address these environmental challenges. By transitioning to electric vehicles (EVs), cities can significantly reduce harmful emissions and improve air quality \citep{MUNOZ2022115412}. (Note that the terms \emph{vehicle} and \emph{bus} are used interchangeably in this paper). Nevertheless, the adoption of electric buses comes with its own set of challenges. One major concern is the higher upfront cost of EVs compared to conventional DVs \citep{MUNOZ2022115412}. This cost disparity can impose financial barriers, particularly when there is a need to replace a large number of buses in existing fleets. Driving range, long charging time, and electricity grid impact of EVs are other issues to be tackled. Although technological advancements have improved battery capacity and charging speeds, EVs still have a shorter range and longer downtime compared to DVs. This can pose operational challenges, especially for longer routes that require long periods of operation. To overcome these challenges, one potential solution is to increase the number of buses in operation. However, the high cost of electric buses can be a hindrance. Therefore, optimizing EV scheduling becomes essential to minimize the bus fleet size and idle time, while ensuring sufficient recharging during idle periods. 

The vehicle scheduling problem (VSP) involves the creation of \emph{vehicle blocks} (hereafter called blocks) based on a set of timetabled service or revenue trips, called \emph{trips}. These trips come with essential spatio-temporal information, including their origin (first stop), destination (last stop), start time, and end time. The objective of the VSP is to strategically organize these trips into blocks that optimize the utilization of vehicles and ensure efficient transit operations. The VSP has been extensively studied for many years, and various solution approaches have been proposed to address its complexity. See \citet{bunte2009overview} and \citet{freling2001models} for comprehensive reviews. However, with the emergence and early adoption of electric vehicles (EVs), there is a need to revisit the problem and adapt it to accommodate the unique characteristics and requirements of EVs.  

The VSP can be classified into two main types based on the number of depots involved: single-depot VSP (SDVSP) and multi-depot VSP (MDVSP). In this study, we consider the single depot electric VSP (SDEVSP). We develop an optimization-based scheduling framework using mixed-integer linear programming (MILP) that can generate bus itineraries (hereafter called \emph{runs}) with given trips. The framework ensures that each trip can be successfully completed using DVs or EVs, and schedules recharging (when necessary) either between blocks or during off-hours. Moreover, we consider next day operability, that is bus runs are created in a way to allow a bus to serve runs in the upcoming planning horizons (often measured in days).

The SDEVSP is known to be NP-hard due to the presence of time or distance limitations, as demonstrated by \citet{bodin1983routing}. This implies that finding optimal solutions to large-scale problems is not computationally feasible within a reasonable time. Yet, many problem instances in urban areas are large-scale. To address this challenge, we propose a two-step solution approach for the SDEVSP. In the first step, solving an integer programming (IP) model, we generate blocks using the SDVSP model presented in \citet{cokyasar2023electric}. Each block is defined as a sequence of consecutive trips and has a designated depot as its starting and ending location. While generating these blocks, control parameters are used in the SDVSP modeling to obtain blocks with shorter than EV range. Since the SDVSP with time or distance constraints is NP-hard, we do not impose hard constraints on block length or time. Yet, the approach outlined in \cref{methodology} provides blocks complying by electrification constraints. Once the blocks are generated, the subsequent step involves chaining them together to form a bus run. This process entails solving the block chaining problem (BCP) using the blocks obtained in the first step. The goal of the BCP is to optimize the combinations of blocks in a way that enables the vehicle to recharge at the depot between two consecutive blocks and after the final block. Additionally, the generated bus runs must satisfy the next day operability constraints, ensuring continuity of operations.

The motivation behind this study is threefold. First, our approach builds upon the widely adopted SDVSP modeling used by transit agencies to create schedules for conventional DVs. By leveraging this well-established method, we facilitate the adoption and implementation of our proposed solution framework, enabling transit agencies to seamlessly transition to electric bus fleets. Transit agencies acknowledge the necessity of creating shorter blocks to electrify bus fleets, and our approach allows for creating these blocks to form efficient electric vehicle (EV) schedules. Second, the SDEVSP is recognized as an NP-hard problem, making it analytically challenging to solve at a large-scale. By breaking down the SDEVSP into the SDVSP and the BCP, we effectively manage the challenges associated with large-scale instances of the SDEVSP. Third, our model incorporates next day operability constraints. This consideration ensures that the scheduling of bus blocks allows for their repetition on the following day, promoting efficient and reliable schedules. Our study considerably advances the SDEVSP literature by providing these contributions. Our practical and scalable solution approach enhances the feasibility and effectiveness of electric bus scheduling, supporting the transition towards more sustainable and environmentally friendlier public transportation systems.

In \cref{lit_rev}, we begin with providing a literature review on the SDEVSP. \cref{methodology} formally describes the problem, the next day operability constraints, and the formulation of the MILP model. \cref{heuristic} outlines a heuristic approach to address the scalability concern in the BCP model. In \cref{num_exp}, we detail the experimental design and the parametric choices, and demonstrate the results of numerical experiments conducted to evaluate the performance of the proposed solution approaches. Finally, \cref{conclusion} concludes the study by summarizing the key findings and discussing potential future research directions.

\section{Literature review}\label[sec]{lit_rev}
Transit service design can be summarized as a sequence of five systematic decisions: Network design, frequency setting, timetabling, vehicle scheduling, and crew scheduling \citep{CEDER1986331}. While many studies in the literature focus on solving these problems separately, some select a subset and solve that selection jointly. See \citet{GUIHAIRE20081251} for a thorough review on these problems. This study solely focuses on the vehicle scheduling problem, i.e. the route alignments, frequencies, and timetables are given and fixed. Similarly, crew scheduling that is solved either after or jointly with vehicle scheduling is also beyond the scope of this study. \cref{literature_tab} gives an overview of the existing relevant literature on the EVSP. 

The existing literature on the electric VSP (EVSP) can be viewed in two main categories based on the number of depots included: SDEVSP and multi-depot EVSP (MDEVSP). While both variants are significant, recent studies have shown a growing interest in the MDEVSP \citep{adler2017vehicle, DIEFENBACH2023828, Li2020, LIU2020118, WEN201673, WU2022322, YAO2020101862, Zhang_Li_Tu_Dong_Chen_Gao_Liu_2021}. For instance, \citet{WU2022322} proposed a branch-and-price method for addressing the MDEVSP, incorporating time-of-use electricity tariffs and peak load risk. Similarly, \citet{DIEFENBACH2023828} employed a branch-and-check method, considering non-linear charging and partial charging to minimize the electric vehicle fleet size in the MDEVSP context. However, in our research, we specifically concentrate on the single depot aspect of the EVSP. This decision is motivated by our understanding of the needs and requirements of large-scale transit agencies. Those agencies that operate out of multiple depots already have their blocks and runs assigned to certain depots by either solving an MDVSP, or by pre-assigning routes or trips to certain depots and solving multiple SDVSPs. In this study, we treat the existing assignment of trips to depots as initial conditions. By focusing on the SDEVSP, we aim to provide practical and applicable solutions that align with the operational context of these agencies. While the MDEVSP is undoubtedly an important area of research, addressing the complexities associated with multiple depots falls beyond the scope and considerations of our study. 

\begin{table}[!htb]
  \scriptsize
  \setlength\extrarowheight{4pt}
  \caption{Summary of the existing relevant literature.}\label[tab]{literature_tab}
  \resizebox{\textwidth}{!}{\begin{tabular}{>{\raggedright}p{0.15\linewidth} p{0.28\linewidth}lcccccccccc}
  \toprule
  Study & Objective & Model & (i) & (ii) & (iii) & (iv) & (v) & (vi) & (vii) & (viii) & (ix) & (x)\\
  \midrule

    \citet{ALWESABI2020118855} & Minimize the battery cost and charging infrastructure costs. & MIQCP & - & - & - & - & - & \checkmark & - & - & -& -\\
   
    \citet{CHAO20132725} & Minimize the capital investment for the electric fleet and the total charging demand. & MILP & - & - &  - & - & - & - & - & - & -& -\\

    \citet{DIEFENBACH2023828} & Minimize the number of vehicles. & MILP & \checkmark & - & - & - & - & - & - & \checkmark & \checkmark & -\\

    \citet{Li2020} & Minimize the total cost of constructing and operating the electric bus system. & MILP & \checkmark & - & - & - & - & - & - & \checkmark & - & -\\
    
    \citet{LIU2020118} & Minimize the number of vehicles. & IP & \checkmark & - & - & - & - & - & - & \checkmark & \checkmark & -\\

    \citet{PERUMAL2021105268}& Minimize the investment costs for vehicles and operational costs. & MILP & - & - & \checkmark & - & - & - & - & - & - & -\\ 
    
    \citet{RINALDI2020102070} & Minimize the total operational cost. & MILP & - & - & - & \checkmark & - & - & \checkmark & - & -& -\\
    
    \citet{SISTIG2023120915} & Minimize the investment costs for vehicles and operational costs. & MILP & - & - & \checkmark & - & - & - & - & - & - & - \\   

    \citet{WEN201673} & Minimize the number of buses and the total traveling distance. & MILP & \checkmark & - & - & - & - & - & - & \checkmark & -& -\\
    
    \citet{WU2022322} & Minimize the total operation cost. & MILP & \checkmark & - & - & - & \checkmark & - & - & - & - & -\\

    \citet{XU2023104057} & Maximize the difference between the profit from the bus fare and the operational cost. & IP & - & \checkmark & - & - & - & - & - & - & - & -\\

    \citet{YAO2020101862}& Minimize the vehicle purchasing cost and operation cost. & IP & \checkmark & - & \checkmark & - & - & - & - & \checkmark & - & -\\

    \citet{Zhang_Li_Tu_Dong_Chen_Gao_Liu_2021} &  Minimize the vehicle purchasing cost and operation cost. & MILP & \checkmark & - & - & \checkmark & - & - & - & \checkmark & \checkmark & -\\

    This study & Minimize the number of vehicles and deadheading time. & MILP & - & - & - & - & - & - & - & \checkmark & - & \checkmark\\
  \bottomrule
  \end{tabular}}
(i) Multiple depots,
(ii) Timetabling,
(iii) Crew scheduling,
(iv) Mixed fleet,
(v) Power grid,
(vi) Placement of charging infrastructure,
(vii) Number of chargers,
(viii) Partial charging,
(ix) Non-linear charging,
(x) Operational continuity,
MIQCP: Mixed-integer quadratically-constrained program.
\end{table}

The SDEVSP has received limited attention in the existing literature, with a few studies dedicated to exploring its various aspects \citep{ALWESABI2020118855, CHAO20132725, XU2023104057, PERUMAL2022395, RINALDI2020102070, SISTIG2023120915}. For instance, \citet{XU2023104057} focused on jointly solving the electric bus timetabling and scheduling problem. They tackled this problem by employing the Lagrangian relaxation heuristic method as their solution approach. It should be noted that including timetabling introduced scalability challenges to their solution method. In our study, timetables are given and fixed, and we ensure that all the revenue trips are served by a vehicle. This deliberate choice ensures that our model is applicable to large-scale problems and can be effectively solved. \citet{li2014transit} studied the fast charging (or battery swapping) in the SDEVSP context. The study considered limited charger capacity at the depot resulting in an NP-hard problem. Eventually, heuristics were developed to solve the problem. As the difficulty of the problem is acknowledged in this study, we solve the problem in two steps. Another related study conducted by \citet{SISTIG2023120915} explored the integrated problem of electric vehicle and crew scheduling. To solve this problem, they employed a metaheuristic based on adaptive large neighborhood search (ALNS). Similarly, \citet{PERUMAL2021105268} also addressed the integrated electric vehicle and crew scheduling problem and utilized an ALNS as their solution approach. Our study does not consider the crew scheduling but considers the operational continuity. By focusing on operational continuity, our research aims to contribute to the field of sustainable electric vehicle scheduling. We recognize the importance of maintaining a consistent and efficient electric vehicle fleet, thereby enabling smoother and more reliable transportation services. \citet{abdelwahed2020evaluating} proposed MILPs to model the problem considering time-dependent electric prices and minimizing the impact on grid. While the grid impact is especially vital, we do not consider it in this study for simplicity.

\section{Methodology}\label[sec]{methodology}

In this section, we provide a formal description of the SDVSP model as presented by \cite{cokyasar2023electric}. We describe how the SDVSP solution method can be used to solve the SDEVSP, and introduce an MILP model to solve the BCP. To ease reading, we adopt a specific notation convention where calligraphic letters denote sets, uppercase Roman letters represent parameters, lowercase Roman letters represent variables and indices, and lowercase Greek letters as superscripts modify parameters and variables.

\subsection{Single Depot Vehicle Scheduling Problem (SDVSP)}\label[sec]{SDVSP}

Let $\sT$ represent set of timetabled bus trips, which are the movements of a bus to serve customers with known origin $O_i$ (first stop) and destination $D_i$ (last stop). The tuple set $\sL$ denotes all feasible arcs that connect bus trips, allowing them to be performed sequentially. Additionally, $\sR = \sL \bigcup \left(s \times \sT\right)\bigcup \left(\sT \times t\right)$ denotes set of all feasible arcs, where $s$ and $t$ indices denote the depot buses are dispatched from and return to, respectively. We denote the \emph{deadheading} time by $T^\tau_{ij}$, that is the travel time from the last stop $D_i$ of trip $i \in \sT$ or from the depot $s$ to the first stop $O_j$ of trip $j \in \sT$ or to the depot $t$. The idle time spent between two consecutive trips is called \emph{layover time}, denoted by $T^\lambda_{ij}$. Note that the layover time does not include the deadheading time but is the time spent after a bus finishes deadheading to the first stop $O_j$ of trip $j$ until the beginning of trip $j$. The block generation cost in time units is defined by $K$, and a unitless weight parameter $W$ adjusts the balance between vehicle costs and layover time. The binary decision variable $l_{ij} = 1$ represents whether trip $j \in \sT$ is served after trip $i \in \sT$, and 0 otherwise. \cref{SDVSP-sets-params-vars} denotes sets, parameters, and variables used in this section, and the mathematical model is in \labelcref{obj_fun_SDVSP}--\labelcref{chain_2_SDVSP}.

\begin{table}[!htb]
  \footnotesize
  \caption{Sets, parameters, and variables used in the SDVSP.}\label[tab]{SDVSP-sets-params-vars}
  \begin{tabularx}{\linewidth}{lX}
  \toprule
    \textbf{Set} & \textbf{Definition}\\
    \midrule
    $\sL$ & set of arcs connecting two consecutive trips\\
    $\sR$ & set of all feasible arcs connecting two consecutive trips, $\sR = \sL \bigcup \left(s \times \sT\right)\bigcup \left(\sT \times t\right)$, where $s$ and $t$ indices denote the depot buses are dispatched from and return to, respectively\\
    $\sT$ & set of timetabled bus trips\\
    \midrule
    \textbf{Parameter} & \textbf{Definition}\\
    \midrule
    $D_i$ & last stop of trip $i \in \sT$\\
    $K$ & a big number representing the block generation cost in time units\\
    $O_i$ & first stop of trip $i \in \sT$\\
    $T_i^\omicron$ & start time of trip $i\in \sT$\\
    $T_i^\rho$ & end time of trip $i\in \sT$\\
    $T_{ij}^\tau$ & deadheading time, the travel time from the last stop $D_i$ of trip $i \in \sT$ to the first stop $O_j$ of trip $j \in \sT$\\
    $T_{ij}^\lambda$ & layover time, the idle time spent between two consecutive trips $i \in \sT$ and $j \in \sT$ at the first stop $O_j$ of trip $j \in \sT$\\
    $W$ & weight factor for layover time between two consecutive trips\\
    \midrule
    \textbf{Variable} & \textbf{Definition}\\
    \midrule
    $l_{ij}$ & $\begin{cases}
          1 & \text{if trip $j \in \sT$ is served after trip $i \in \sT$, $i \neq j$}\\
          0 & \text{otherwise}\\
          \end{cases}$\\
    \bottomrule
  \end{tabularx}
\end{table}

\begin{equation}\label[objfunc]{obj_fun_SDVSP}
    \min \sum_{(i,j)\in\sL} (T^\tau_{D_i O_j} + W T^\lambda_{ij})l_{ij} + \sum_{j\in\sT}\Big(K+T^\tau_{s O_j}\Big)l_{sj} + \sum_{i\in\sT}T^\tau_{D_i t}l_{it}
\end{equation}
\noindent subject to,
\begin{equation}\label[consset]{chain_1_SDVSP}
    \sum_{j:\left(i,j\right)\in \sR}l_{ij} = 1 \qquad \forall i\in\sT
\end{equation}
\begin{equation}\label[consset]{chain_2_SDVSP}
    \sum_{i:\left(i,j\right)\in \sR}l_{ij} = 1 \qquad \forall j\in\sT
\end{equation}
\begin{equation*}\label[consset]{non_neg_SDVSP}
    \nonumber l_{ij} \in \{0,1\} \qquad \forall \left(i,j\right)\in \sR
\end{equation*}

The objective function \labelcref{obj_fun_SDVSP} is to minimize the weighted summation of the total non-revenue time (deadheading and weighted layover times) and the fleet size by adding an artificial time $K$ to depot-to-trip travels. Constraints \labelcref{chain_1_SDVSP} and \labelcref{chain_2_SDVSP} guarantee that each trip follows exactly one preceding trip or a depot trip and is subsequently followed by exactly one subsequent trip or a depot trip. Adjusting the parameters $K$ and $W$ affect the block length. Increasing the value of $K$ leads to longer blocks as the block generation cost becomes more significant, while increasing $W$ results in shorter blocks since the importance of layover time increases in relation to deadheading time and block generation cost.

\subsection{Block Chaining Problem (BCP)}\label[sec]{BCP}
The BCP is to find the optimal combination of bus blocks to be served consecutively by EVs that minimize the total \emph{depot layover time} and the number of EVs, while making use of the depot layover time between blocks for recharging. We assume that each depot is sufficiently large to accommodate new buses and the charging equipment. Moreover, there are as many slow and fast chargers as needed, resulting in zero waiting times for recharging. Note that the depot layover time in this section is different from the layover time in the previous subsection. While layover time in SDVSP is the time spent at a trip origin until the start of a trip, the depot layover time is the time spent at the depot between two blocks of a given vehicle. To solve the BCP, it is necessary to satisfy certain constraints related to the blocks. The blocks must adhere to EV range constraints, ensuring that the distance or time of each block does not exceed the EV's range. This is facilitated by adjusting the parameters $K$ and $W$ described previously. As mentioned in \cref{intro}, this parametric approach does not guarantee that all blocks are within the EV range since we do not have hard constraints. However, an \emph{acceptable} or \emph{targeted} share of within-range blocks can be obtained using this soft approach. Temporal conditions between blocks must also be met. For instance, the start time of the succeeding block should be later than the end time of the preceding block. Now, we formally describe an MILP formulation to solve the BCP, building upon the block results obtained from solving the SDVSP.

The set of blocks that can be run by EVs (i.e., blocks meeting the range constraints) is denoted by $\sB$. The tuple set $\sE$ denotes all feasible arcs that connect bus blocks within the planning horizon. The tuple set $\sC$ denotes all feasible arcs that connect bus blocks of consecutive horizons, that is each pair consists of a bus block from the current planning horizon and a bus block from the next planning horizon, and they can be combined in a sequential order. Furthermore, $\sA = \sE \bigcup \left(s \times \sB\right)\bigcup \left(\sB \times t\right)$ denotes set of all feasible arcs within a given horizon. Similar to the SDVSP, the indices $s$ and $t$ indicate the depot from which buses are dispatched and the depot to which they return, respectively. \cref{sets_params} provides sets and parameters used in the MILP to solve the BCP.

\begin{table}[!htb]
  \footnotesize
  \caption{Sets and parameters used in the MILP.}\label[tab]{sets_params}
  \begin{tabularx}{\linewidth}{lX}
  \toprule
    \textbf{Set} & \textbf{Definition}\\
    \midrule
    $\sA$ & set of all feasible arcs connecting two consecutive blocks within the horizon, $\sA = \sE \bigcup \left(s \times \sB\right)\bigcup \left(\sB \times t\right)$, where $s$ and $t$ indices denote the depot buses are dispatched from and return to, respectively\\
    $\sB$ & set of timetabled bus blocks\\
    $\sC$ & set of arcs connecting two consecutive blocks over night $i$ (first, current horizon) and $j$ (second, next horizon), $\sC = \Big\{\left(i, j\right) | i, j \in \sB \land L \leq \left(\overline{T} + T_j^\alpha-T_i^\beta\right) \leq U \Big\}$\\
    $\sE$ & set of arcs connecting two consecutive blocks $i$ (first, current horizon) and $j$ (second, current horizon), $\sE = \Big\{\left(i, j\right) | i, j \in \sB \land i \neq j \land L \leq \left(T_j^\alpha-T_i^\beta\right) \leq U \Big\}$\\
    \midrule
    \textbf{Parameter} & \textbf{Definition}\\
    \midrule
    $\overline{B}$ & battery capacity measured in time units\\
    $B_i$ & energy consumption of block $i\in\sB\cup \{s\}$ measured in time units, and $B_s=0$\\
    $K^\prime$ & a big number representing the vehicle cost measured in time units\\
    $L$ & minimum admitted recharging time between two consecutive blocks\\
    $\bigM_1$ & big number, that is $\bigM_1 > \overline{B}+ \max\big\{{(\overline{T}+\max_{i\in \sB}T_i^\alpha) R^\nu, \max_{i\in \sB}R^\delta T_i^\alpha}\big\}$\\
    $\bigM_2$ & big number, that is $\bigM_2 \geq \overline{B}+ 2\bigM_1$ \\
    $R^\delta$ & rate of recharge during day, i.e., energy (in time units) gained by recharging in one unit of time, e.g., $R^\delta$ minutes of driving range is gained by recharging a bus for one minute\\
    $R^\nu$ & rate of recharge during night, i.e., energy (in time units) gained by recharging in one unit of time, e.g., $R^\nu$ minutes of driving range is gained by recharging a bus for one minute\\
    $\overline{T}$ & end of planning horizon in time units\\
    $T_i^\alpha$ & start time of block $i\in \sB$\\
    $T_i^\beta$ & end time of block $i\in \sB$\\
    $U$ & maximum admitted recharging time between two consecutive blocks\\
    $W^\prime$ & weight factor for recharging time between two consecutive blocks\\
    \bottomrule
  \end{tabularx}
\end{table}

The energy consumption to run block $i \in \sB$, measured in time units, is denoted by $B_i$. This energy consumption is assumed to be a linear function of the travel time for the block. The start and end times of block $i \in \sB$, which are obtained by solving the SDVSP, are denoted by $T_i^\alpha$ and $T_i^\beta$, respectively. A recharging between consecutive blocks in the same planning horizon is considered to occur during the day, while recharging between consecutive blocks, one in the current and the other in the next planning horizon, is assumed to be overnight. The rate of recharge during the day is denoted by $R^\delta$, while the rate of recharge overnight is denoted by $R^\nu$. These recharge rates represent the amount of energy (measured in time units) gained by recharging for one unit of time. The battery capacity, which corresponds to the EV range and is measured in time units, is denoted by $\overline{B}$. This parameter represents the maximum amount of energy that the EV's battery can store, determining the maximum duration the vehicle can travel without recharging. The end time of the planning horizon, denoted by $\overline{T}$, establishes the time limit or deadline for the scheduling of blocks. This parameter sets the boundary for the scheduling process, ensuring that all blocks are scheduled within the specified time frame. To control the layover time between consecutive blocks, we introduce the maximum and minimum layover time limits that are denoted by $U$ and $L$, respectively. Weight parameters $K^\prime$ (in time units) and $W^\prime$ (unitless) represent the vehicle cost and importance of layover time against fleet size in the objective function, respectively. Lastly, $\bigM_1$ and $\bigM_2$ are adequately big numbers, where $\bigM_1 > \overline{B}+ \max\big\{{(\overline{T}+\max_{i\in \sB}T_i^\alpha) R^\nu, \max_{i\in \sB}R^\delta T_i^\alpha}\big\}$ and $\bigM_2 \geq \overline{B}+ 2\bigM_1$. 

Binary decision variable $y_{ij}=1$ if block $j \in \sB$ is served after block $i \in \sB$, and $y_{ij}=0$, otherwise. Binary decision variable $z_{ij}$ takes a value of 1 if block $j \in \sB$ on the next day can be served after block $i \in \sB$ in the current day, and 0 otherwise. Decision variable $v_{ij} \in \mathbb{R}_{\geq 0}$ represents the state of charge (SOC) in time units at the beginning of block $j \in \sB$ after serving block $i \in \sB$.
Decision variable $v_{ij}^\prime \in \mathbb{R}$ represents the SOC in time units at the beginning of block $j \in \sB$ on the next planning horizon after serving block $i \in \sB$ in the current horizon. Decision variable $b_i \in \mathbb{R}_{\geq 0}$ denotes the SOC in time units at the beginning of block $i \in \sB$. Decision variable $u_{ij} \in \mathbb{R}_{\geq 0}$ represents the amount of energy gained measured in time units during the layover time between two consecutive blocks. Additionally, we introduce two auxiliary binary variables: $x_{ij}$ and $n_{ij}$. These variables are used to linearize the max and min functions, respectively. \cref{vars} provides variables and variable definitions used in the MILP. The mathematical model to solve the BCP is as follows:

\begin{table}[!htb]
  \footnotesize
  \caption{Variables used in the MILP.}\label[tab]{vars}
  \begin{tabularx}{\linewidth}{lX}
  \toprule
    \textbf{Variable} & \textbf{Definition}\\
    \midrule
    $b_{i}$ & state of charge at the beginning of block $i\in\sB\cup\{s\}$ measured in time units, $b_{i}\in\mathbb{R}_{\geq 0}$\\
    $n_{ij}$ & auxiliary binary variable used to linearize the min function, $(i,j) \in \sC$\\
    $u_{ij}$ & energy gained between blocks $i$ and $j$ measured in time units on current horizon, $u_{ij}\in\mathbb{R}_{\geq 0},~(i,j) \in \sA$\\
    $v_{ij}$ & state of charge at the beginning of block $j$ on current horizon after serving block $i$ on current horizon measured in time units, $v_{ij}\in\mathbb{R}_{\geq 0},~(i,j) \in \sA$\\
    $v_{ij}^\prime$ & state of charge at the beginning of block $j$ on next horizon after serving block $i$ on current horizon measured in time units, $v_{ij}^\prime\in\mathbb{R},~(i,j) \in \sC$\\
    $x_{ij}$ & auxiliary binary variable used to linearize the max function, $(i,j) \in \sA$\\
    $y_{ij}$ & $\begin{cases}
          1 & \text{if block $j$ on current horizon is served after block $i$ on current horizon, $(i,j) \in \sA$}\\
          0 & \text{otherwise}\\
          \end{cases}$\\
    $z_{ij}$ & $\begin{cases}
          1 & \text{if block $j$ on next horizon can be served after block $i$ on current horizon, $(i,j) \in \sC$}\\
          0 & \text{otherwise}\\
          \end{cases}$\\
    \bottomrule
  \end{tabularx}
\end{table}

\begin{equation}\label[objfunc]{obj_fun}
    \min \sum_{\substack{{\left(i,j\right)}\in \sE}} W^\prime\left(T_j^\alpha-T_i^\beta\right)y_{ij}+\sum_{i\in\sB}K^\prime y_{si}
\end{equation}
\noindent subject to,
\begin{equation}\label[consset]{chain_1}
    \sum_{j:\left(i,j\right)\in \sA}y_{ij} = 1 \qquad \forall i\in\sB
\end{equation}
\begin{equation}\label[consset]{chain_2}
    \sum_{i:\left(i,j\right)\in \sA}y_{ij} = 1 \qquad \forall j\in\sB
\end{equation}
\begin{equation}\label[cons]{battery_state_from_i_to_j}
    v_{ij} = \max \Big\{b_i - B_i - \bigM_1\left(1-y_{ij}\right)+u_{ij}, 0\Big\} \qquad \forall (i,j) \in \sA
\end{equation}
\begin{equation}\label[cons]{battery_state_before_j}
    b_{j} = \sum_{i:\left(i,j\right)\in \sA}v_{ij} \qquad \forall j\in\sB
\end{equation}
\begin{equation}\label[cons]{battery_bounds}
    B_i \leq b_i \leq \overline{B} \qquad \forall i\in \sB
\end{equation}
\begin{equation}\label[cons]{charging_time_limit}
    u_{ij}\leq \left(T_j^\alpha-T_i^\beta\right)R^{\delta}y_{ij} \qquad \forall {\left(i,j\right)}\in\sE
\end{equation}
\begin{equation}\label[cons]{charging_time_before_depot}
    u_{it} = 0 \qquad \forall i\in\sB
\end{equation}
\begin{equation}\label[cons]{battery_state_from_overnight_i_to_j}
    v_{ij}^\prime = \min \Big\{\overline{B}, v_{it}+\left(\overline{T} + T_j^\alpha-T_i^\beta\right)R^{\nu} - \bigM_1\left(2-y_{sj}-y_{it}\right)\Big\} \qquad \forall \left(i,j\right)\in \sC
\end{equation}
\begin{equation}\label[cons]{overnight_connection_1}
    v_{ij}^\prime \geq b_j -\bigM_2 \left(1-z_{ij}\right) \qquad \forall \left(i,j\right)\in \sC
\end{equation}
\begin{equation}\label[cons]{continuation_possibility_1}
    \sum_{i:\left(i,j\right)\in \sC}z_{ij} = y_{sj} \qquad \forall j\in \sB 
\end{equation}
\begin{equation}\label[cons]{continuation_possibility_2}
    \sum_{j:\left(i,j\right)\in \sC}z_{ij} = y_{it} \qquad \forall i\in \sB 
\end{equation}
\begin{equation*}\label[consset]{non-neg}
    \nonumber y_{ij}, ~z_{ij} \in \{0,1\},  b_{i}, ~u_{ij}, ~v_{ij}\in \mathbb{R}_{\geq 0}, v_{ij}^\prime\in \mathbb{R}.
\end{equation*}

The objective function \labelcref{obj_fun} minimizes the total layover time between blocks and the number of vehicles. Constraints \labelcref{chain_1} and \labelcref{chain_2} guarantee that each block follows exactly one preceding block or a depot block and is followed by exactly one subsequent block or a depot block, respectively. Constraints \labelcref{battery_state_from_i_to_j} and \labelcref{battery_state_before_j} determine the SOC at the start of block $j \in \sB$ based on the SOC at the beginning of the preceding block $i \in \sB$, the energy consumption during block $i$, and the energy gained between the blocks $i$ and $j$. Note that $v_{sj}=b_s+u_{sj}$ when $y_{sj}=1$ (that is for each run) in \labelcref{battery_state_from_i_to_j}. Then, we know from \labelcref{battery_state_before_j} that $b_j=b_s+u_{sj}$. Therefore, initial SOC of each run can vary by artificially charging from $s$ to the first block in the current horizon, and we can ensure that the initial SOC for each run is a variable. Constraints \labelcref{battery_bounds} ensure that the SOC does not exceed the maximum battery capacity, and  EVs have adequate SOC to complete each block $i\in\sB$ without running out of energy. Constraints \labelcref{charging_time_limit} guarantee that the energy gained between two consecutive blocks does not exceed the maximum amount of energy that can be gained during the layover time between those blocks. Constraints \labelcref{charging_time_before_depot} enforce that no daytime charging takes place if block $i$ is the last block of the horizon. 

The set of constraints \labelcref{battery_state_from_overnight_i_to_j} - \labelcref{continuation_possibility_2} account for the feasibility of the next horizon's operations. Constraints \labelcref{battery_state_from_overnight_i_to_j} calculate the SOC at the beginning of block $j \in \sB$ on the next horizon, after serving block $i \in \sB$ on the current horizon. This calculation takes into account the SOC at the end of block $i$ and considers the amount of energy gained between blocks $i$ and $j$, where $\left(i, j\right) \in \sC$. To establish the connection between block $j \in \sB$ on the next horizon and block $i \in \sB$ on the current horizon, we introduce constraint \labelcref{overnight_connection_1}. This constraint ensures that if $z_{ij}=1$, the SOC $v_{ij}^\prime$ at the beginning of block $j$ is sufficient to serve that block. Constraints \labelcref{continuation_possibility_1} ensure that each block $j$ that start from the depot on the next horizon is preceded only by one block $i$. Similarly, constraints \labelcref{continuation_possibility_2} ensure that each block $i$ that ends at the depot on the current horizon is succeeded by only one block in the next horizon.

Constraints \labelcref{battery_state_from_i_to_j} and \labelcref{battery_state_from_overnight_i_to_j} in their current form involve min and max functions with variables, which is quite straightforward to deal with by many commercial solvers without the need for linearization. However, it can still be useful to remove the non-linearity to possibly accelerate the solution. To this end, we replace constraints \labelcref{battery_state_from_i_to_j} with the set of constraints \labelcref{battery_state_from_i_to_j_l1} - \labelcref{battery_state_from_i_to_j_l3} and constraints \labelcref{battery_state_from_overnight_i_to_j} with the set of constraints \labelcref{battery_state_from_overnight_i_to_j_l1} - \labelcref{battery_state_from_overnight_i_to_j_l4}. 

\begin{equation}\label[cons]{battery_state_from_i_to_j_l1}
    v_{ij} \geq b_i - B_i + u_{ij} - \bigM_1\left(1-y_{ij}\right) \qquad \forall (i,j)\in \sA
\end{equation}
\begin{equation}\label[cons]{battery_state_from_i_to_j_l2}
    v_{ij} \leq b_i - B_i + u_{ij} - \bigM_1\left(1-y_{ij}-x_{ij}\right)\qquad \forall (i,j)\in \sA
\end{equation}
\begin{equation}\label[cons]{battery_state_from_i_to_j_l3}
    v_{ij} \leq \bigM_1 \left(1-x_{ij}\right) \qquad \forall (i,j)\in \sA
\end{equation}
\begin{equation}\label[cons]{battery_state_from_overnight_i_to_j_l1}
    v_{ij}^\prime \leq \overline{B} \qquad \forall \left(i,j\right)\in \sC
\end{equation}
\begin{equation}\label[cons]{battery_state_from_overnight_i_to_j_l2}
    v_{ij}^\prime \leq v_{it}+\left(\overline{T} + T_j^\alpha-T_i^\beta\right)R^{\nu} - \bigM_1\left(2-y_{sj}-y_{it}\right) \qquad \forall \left(i,j\right)\in \sC
\end{equation}
\begin{equation}\label[cons]{battery_state_from_overnight_i_to_j_l3}
    v_{ij}^\prime \geq \overline{B} - \bigM_2 n_{ij} \qquad \forall \left(i,j\right)\in \sC
\end{equation}
\begin{equation}\label[cons]{battery_state_from_overnight_i_to_j_l4}
    v_{ij}^\prime \geq v_{it}+\left(\overline{T} + T_j^\alpha - T_i^\beta\right)R^{\nu} - \bigM_1\left(3-y_{sj}-y_{it}-n_{ij}\right) \qquad \forall \left(i,j\right)\in \sC
\end{equation}
\begin{equation*}\label[consset]{non-neg2}
    \nonumber x_{ij}, ~y_{ij}, ~n_{ij} \in \{0,1\},  b_{i}, ~u_{ij}, ~v_{ij}\in \mathbb{R}_{\geq 0}, v_{ij}^\prime\in \mathbb{R}.
\end{equation*}

\section{Heuristic solution approaches}\label[sec]{heuristic}
The SDEVSP is a known NP-hard problem, presenting a computational challenge for finding an optimal solution. To tackle this problem, we propose a two-step heuristic solution approach. In the first step, our objective is to identify suitable values for $K$ and $W$ that enable the generation of blocks where the service time for each block does not exceed the vehicle range. We use the method proposed by \citet{cokyasarieee} to determine the suitable values for $K$ and $W$, then we use the resulting blocks generated in the previous step and solve the BCP. The BCP aims to generate an optimal schedule for electric vehicles based on the given blocks and their associated start and end times, taking into account charging requirements. Since the BCP is a variant of the SDVSP model with resource constraints, it is NP-hard. We introduce two solution algorithms to solve large-scale instances: A divide-and-conquer (DaC) algorithm and a greedy heuristic algorithm. In the following Sections \labelcref{sol_BCP_1} and \labelcref{sol_BCP_2}, we elaborate on these solution algorithms in detail. Additionally, we present a computational analysis in \cref{computational_performance}, where we evaluate and compare the performance of the DaC and the greedy heuristic algorithm with an MILP solver.

\subsection{Divide-and-conquer (DaC) algorithm }\label[sec]{sol_BCP_1}
Divide and conquer (DaC) is well-known algorithm \citep{alma99271453112005897}. The idea behind the DaC algorithm is to break down the large-scale BCP into smaller and manageable subproblems. The subproblems at adequately small size (e.g., 20 blocks) can be solved independently using commercial solvers. Once the subproblems are solved, the solutions are combined to form an overall solution for the master problem. This combination step ensures that the solution obtained is feasible as \cref{lemma} denotes.

\begin{lemma}\label[lem]{lemma}
Let $\mathcal{M}$ be the master problem and $\sP$ be a set of subproblems, that $\mathcal{M} = \bigcup_{p \in \sP}p$. If each subproblem $p \in \sP$ has a feasible solution $X_p$, then $X_\mathcal{M} = \bigcup_{p \in \sP}X_p$ is a feasible solution for $\mathcal{M}$.
\end{lemma}

\begin{proof}[\textit{Proof of} \cref{lemma}]
If each subproblem $p \in \sP$ has a feasible solution $X_p$, then it means that blocks in $\sB_p$ form one or more feasible bus runs. Each run in a given subproblem $p \in \sP$ on the current horizon is connected to at least one run in the next horizon within the same subproblem $p \in \sP$. Since there are no constraints on charging capacity or garage space, merging the solutions $X_p$ into $X_\mathcal{M}$ is analogous to solving separate problems for separate garages.
\end{proof}

In order to break down the large-scale BCP into smaller subproblems, we employ the Kernighan-Lin (K-L) bisection algorithm, as introduced by \citet{Kernighan_Lin}. To apply the K-L bisection algorithm in the context of the BCP, we begin with representing the problem as a graph. Each block is represented as a vertex, based on the the set $\sB$, and the relationships between the blocks are captured as edges, based on the tuple set $\sE$. The K-L bisection algorithm then aims to partition this graph into two subgraphs with the goal of minimizing the number of edges between the partitions while ensuring approximately equal number of vertices in each partition. This partitioning is achieved through an iterative process that involves swapping vertices between the two partitions to maximize the reduction in the number of edges between them. Using this method, we attempt to minimize the optimality deviation caused by partitioning.

While the K-L algorithm is originally designed to partition a problem into two subproblems, we aim to divide the problem into a larger number of subproblems. To do this, we can repeat the K-L algorithm multiple times. In each iteration, the algorithm partitions a subproblem into two subproblems by dividing the corresponding graph representation. The first iteration applies the K-L algorithm to the original problem, resulting in two subproblems. Subsequent iterations apply the K-L algorithm to each subproblem from the previous iteration, dividing them into two subproblems each. 

Let $\norm{\mathcal{M}}$ represent the number of blocks in the master problem and $\norm{p}$ represent the maximum number of blocks that can be solved using commercial solvers within a reasonable timeframe. The target number of subproblems is then $m = \frac{\norm{\mathcal{M}}}{\norm{p}}$. This would require $n=\ceil{\log_2^m}$ iterations, and in each iteration $k=1,\ldots,n$ the number of partitionings is $2^{k-1}$ resulting in a total of $2^{n}-1$ partitionings. The final number of subproblems would then be $2^{n}$. Note that, ideally, we divide the problem into a reasonably large number of subproblems to fully exploit the potential for parallelism and solution efficiency. However, dividing the problem into a larger number of subproblems leads to a natural decrease in solution quality. Therefore, we carefully consider this trade-off to determine an appropriate value for $m$. 

\subsection{Greedy algorithm}\label[sec]{sol_BCP_2} 

The greedy algorithm in \cref{greedy_algo} iteratively assigns blocks to vehicles in a way that minimizes the number of vehicles needed. The algorithm follows a greedy strategy, making locally optimal decisions at each step. It is one of the traditional methods to solve scheduling problems and is similar to the \emph{earliest due date rule} presented in \citep[pp. 152]{sule2007production}. The algorithm begins with sorting the set of blocks $\sB$, based on their start times $T^\alpha_i$. We initialize various variables, including $b_i$, $u_{ij}$, and $u^\prime_{ij}$, to zero. Additionally, we set $b_{\sB(0)}=\overline{B}$.

Next, we define $b^\phi_i$ as the net energy consumption until the end of block $i$. It is computed by summing the consumption of all preceding blocks in the chain up to the previous block, and subtracting the sum of all charging values $u_{ij}$ that occurred between those blocks. All $b^\phi_i$ values are initially set to zero. To begin, the algorithm generates the first vehicle $\sV_0$ by adding the first block $\sB(0)$ to the itinerary. The net energy consumption for this vehicle is set to the consumption of the first block, $B_{\sB(0)}$. After removing the first block $\sB(0)$ from $\sB$, we initialize the vehicle counter, denoted as $v$, and the block counter, denoted as $k$, to zero.

The algorithm runs until all blocks are assigned to vehicles. Within the while loop, the first condition checks if a vehicle has at least one block already inserted. If not, the first block in the current set of blocks is inserted. For a given block $i$, which represents the last block in the current vehicle, and a block $j$ currently under consideration for insertion, temporal conditions are evaluated. These conditions compare the start time of block $j$ with the start time of block $i$ to ensure chronological feasibility for both during the current horizon and also for the next horizon. Temporary variables $u$, $b$, $u^\prime$, and $b^\phi$, are calculated. If the SOC conditions are met for these variables, indicating sufficient energy levels, the evaluated block is inserted into the current vehicle.

The first condition checks if the overnight charging $u^\prime$ is greater than or equal to the net energy consumption if the current block is inserted into the itinerary. This guarantees that there will be enough battery capacity for the next time period to continue the sequence of blocks. The second condition verifies that the SOC $b_i$ of the last block in the vehicle is greater than or equal to the energy consumption $B_i$ of that block. Lastly, the SOC level after completing block $i$ (i.e., $b_i - B_i$) and charging the vehicle with $u$ is checked to determine if it is sufficient to execute block $j$. If any of the SOC conditions are not satisfied, the block counter $k$ or the vehicle counter $v$ is incremented accordingly. When a block is successfully inserted, it is removed from the set of blocks $\sB$.

\begin{algorithm}[!htb]
\SetArgSty{textnormal}
\scriptsize
\SetKwInOut{Input}{Input}
\SetKwInOut{Output}{Output}
\Input{~ {$\sB$, $B_i$, $\overline{B}$, $\sC$, $R^\delta$, $R^\nu$, $T^\alpha_i$, $T^\beta_i$, $\overline{T}$}}
\Output{~ {$\sV$ \Comment{A set storing blocks of vehicles.}}}
\SetKwFunction{FGreedy}{\textproc{Greedy}}
  \SetKwProg{Fn}{Function}{:}{\KwRet {$\sV$}}
  \Fn{\FGreedy{}}
  {
\textproc{Sort}($\sB$, $T^\alpha$)\;

$b_i, b^\phi_i \gets 0~\forall i\in\sB$;
${u_{ij} \gets 0~\forall (i,j)\in\sA; u^\prime_{ij}} \gets 0~\forall (i,j)\in\sC$;
$b_{\sB(0)} \gets \overline{B}$;
$b^\phi_{0} \gets B_{\sB(0)}$\;

$\sV_0 \gets \{\}$\;
$\sV_0 \gets \sV_0 \cup \{\sB(0)\}$\;
$\sB \gets \sB \setminus \sB(0)$; \Comment{Delete first block.}\\
$v \gets 0$; \Comment{Index of vehicle id.}\\
$k \gets 0$; \Comment{Counter for block indices of $\sB$.}\\
\While{$\sB \neq \{\}$,}{
\uIf{$|\sV_v| \neq 0$,}{
$i \gets \sV_v(|\sV_v|)$; \Comment{Last block in $\sV_v$ is $i$.}\\
$j \gets \sB(k)$; \Comment{The $k^{\text{th}}$ block in $\sB$ is $j$.}\\
\uIf{$i\neq j$ and temporal conditions hold,}{

$u \gets \min\left(\overline{B}-b_i + B_i, (T_j^\alpha-T_i^\beta)R^{\delta}\right)$\;
$b \gets b_i - B_i + u$\;
$u^\prime \gets \min\left(\overline{B} - b + B_j, (\overline{T} + T^\alpha_{\sV_v(0)}-T_i^\beta)R^{\nu}\right)$\;
$b^\phi \gets b^\phi_v + B_j - u^\prime$\;
\uIf{$u^\prime \geq b^\phi \land b_i\geq B_i \land b\geq B_j,$}{
$\sV_v \gets \sV_v \cup \{j\}$; $u_{ij} \gets u$\;
$b_j \gets b$; $u^\prime_{ij} \gets u^\prime$; $b^\phi_v \gets b^\phi$\;
$\sB \gets \sB \setminus \{j\}$\;
}
\Else{
\uIf{$k+1 < |\sB|$,}{$k \gets k + 1$;}
\Else{$v \gets v + 1$; $\sV_v \gets \{\}$; $k \gets 0$;}
}
}
\Else{\uIf{$k+1 < |\sB|$,}{$k \gets k + 1$;}
\Else{$v \gets v + 1$; $\sV_v \gets \{\} $; $k \gets 0$;}}
}
\Else{
$\sV_v \gets \sV_v \cup \{\sB(0)\}$\;
$b_{\sB(0)}\gets \overline{B}$; $b^\phi_{v} \gets B_{\sB(0)}$;
$\sB \gets \sB \setminus \sB(0)$\;
}
}
$\sV \gets \bigcup_{v^\prime \in \{0,1,\ldots,v\}}\sV_{v^\prime}$\;
}
 \caption{Greedy algorithm pseudocode}\label[algo]{greedy_algo}
\end{algorithm}

\section{Numerical Experiments}\label[sec]{num_exp}
In this section, we provide an overview of our experimental design and data in \cref{Design of experiments}, compare the performance and limitations of the Greedy, DaC, and the MILP solver methods in \cref{computational_performance}, and conduct case studies in \cref{case_studies} to reveal key takeaways on large-scale, real-world transit services.

\subsection{Design of experiments}\label[sec]{Design of experiments}
We conducted numerical experiments in Austin, TX and the Chicago Metropolitan Area transit networks. We utilized Capital Metropolitan Transportation Authority (CapMetro) network for Austin, and the Chicago Transit Authority (CTA) and the PACE Suburban Bus networks for Chicago through General Transit Feed Specification (GTFS) data \citep{GTFS}. CapMetro operates 75 bus routes, and Chicago agencies collectively operate a total of 325 bus routes \citep{auld2016polaris}. The locations of these routes and depots
used by these agencies are indicated in \cref{fig:Map}. With the data provided, we identified 17 bus depots in Chicago and verified the number and locations through official websites of the agencies; however, we could locate four bus depots in Austin but could not verify neither the number nor the locations from other sources.
To obtain the necessary trip schedule data, we referred to the GTFS. For CTA trips, we utilized the route-to-depot mapping information available in \citep{Chicagobus} to assign trips to their respective depots. For CapMetro and PACE routes, we did not find any mapping information. Therefore, we computed the mid-point of each trip and assigned it to the closest depot. 

\begin{figure}
\centering
\begin{subfigure}{.5\textwidth}
  \centering
  \includegraphics[width=1\linewidth]{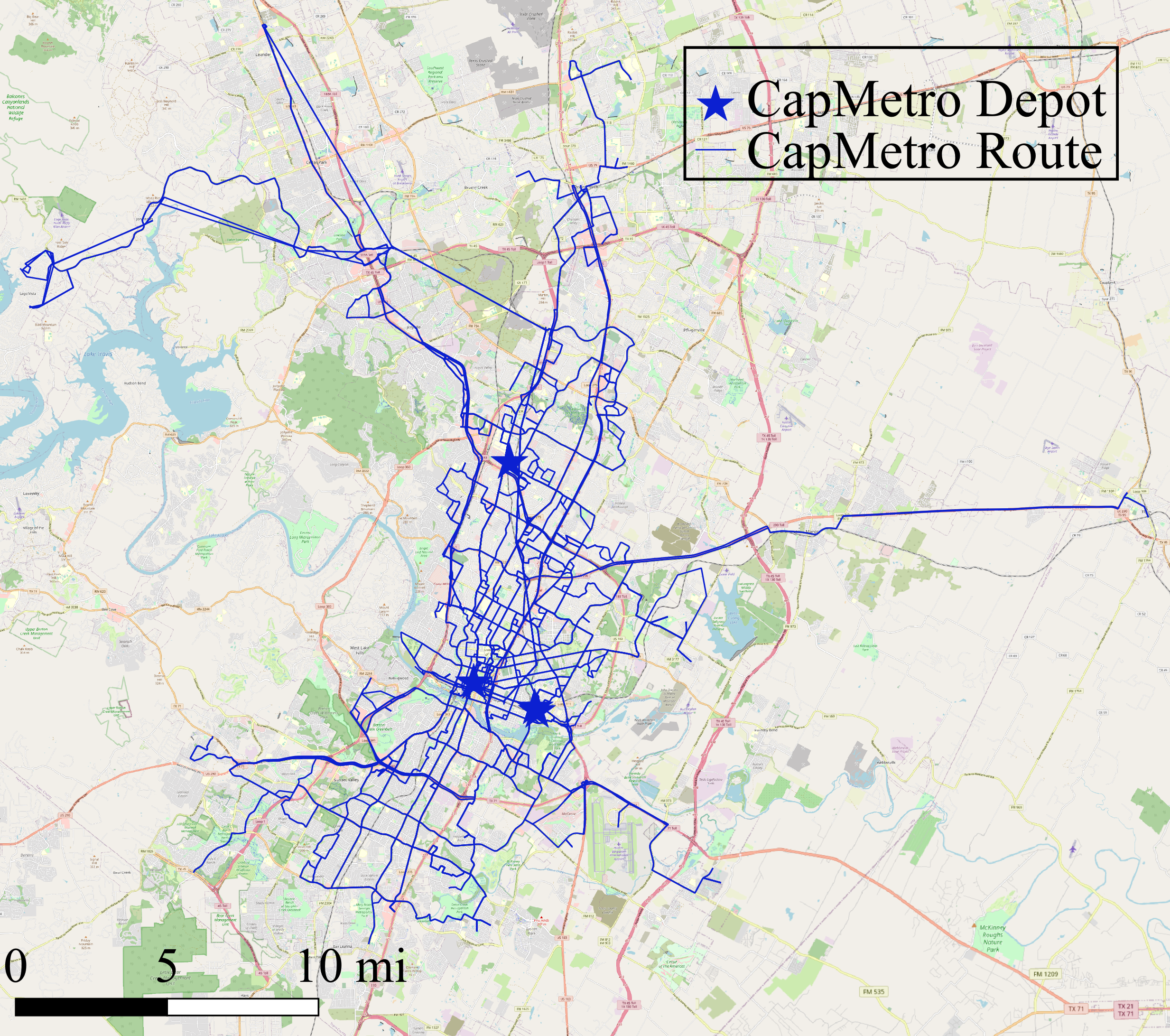}
  \caption{Austin, TX.}
  \label{fig:austin_map}
\end{subfigure}
\quad
\begin{subfigure}{.39\textwidth}
  \centering
  \includegraphics[width=1\linewidth]{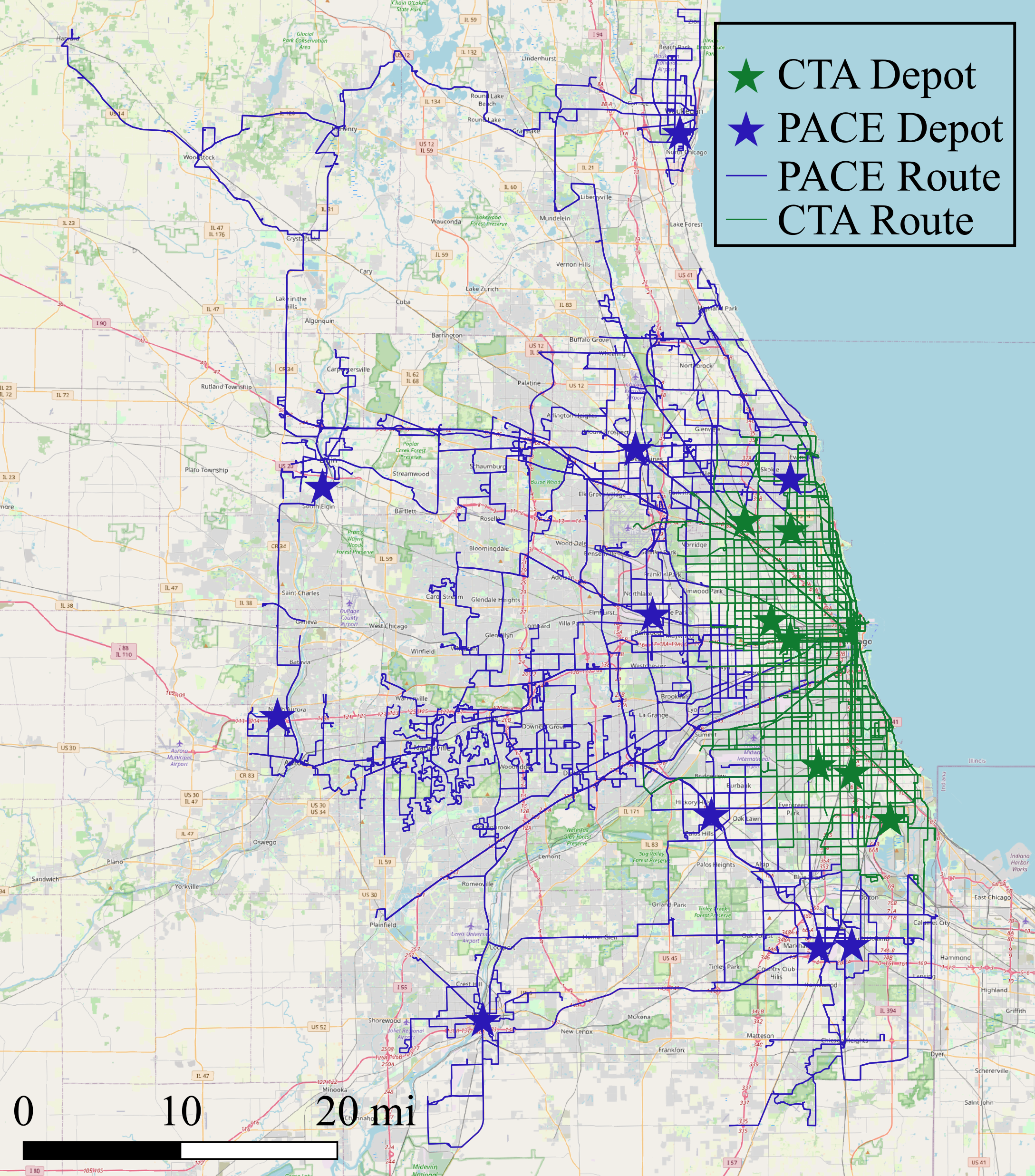}
  \caption{Chicago metropolitan area.}
  \label{fig:chicago_map}
\end{subfigure}
\caption{Maps showing the case study regions and depot locations and routes of the three transit
agencies.}
\label{fig:Map}
\end{figure}

The battery capacity $\overline{B}=120$ minutes. The cost parameter $K^\prime=50,000$ seconds. The values of $\bigM_1=10^6$ and $\bigM_2=3\times 10^6$. The planning horizon $\overline{T}=86,400$ seconds, i.e., 24 hours. The weight factor for recharging time between consecutive blocks $W^\prime=1$. The revenue trip information, i.e. start and end times, and locations, are obtained from the GTFS data \citep{GTFS}. However, data for deadhead travel times is not available, so we assumed an average speed of 30 mph and used Manhattan distances as a basis for estimating deadheading travel times. The battery consumption $B_i$'s are calculated as a summation of revenue trip and deadhead trip travel times within a block. For the parameters $L$ and $U$, we set $L = 0$ and left the upper bound $U$ unrestricted. For our analysis, we consider a 40-foot bus for both DVs and EVs. The assumed vehicle energy consumption rate $E = 220$ kW \citep{CTAreport}. To determine the EV battery capacity, we use the equation $R^\kappa = E \overline{B}$, which yields a battery capacity of 440 kWh for an EV range of two hours. The parameters are summarized in \cref{param_value}.

Regarding the charging infrastructure, we considered both fast charging and slow charging. Overnight charging utilizes slow charging, while daytime charging utilizes fast charging. The power for fast charging is represented by $P^\delta=450$ kW, while slow charging is represented by $P^\nu = 125$ kW. To determine the rates of recharge, we can apply the formulas $R^\delta = \frac{P^\delta}{E}$ and $R^\nu = \frac{P^\nu}{E}$. These calculations yield recharge rates of $R^\delta = 2.045$ and $R^\nu = 0.568$, respectively. In words, charging a bus for one minute overnight or during the day increases the SOC by 0.568 minutes or by 2.045 minutes, respectively.

\begin{table}[!htb]\caption{Parametric values used.}\label[tab]{param_value} 
    \footnotesize
    {\begin{tabular*}\textwidth{c@{\extracolsep{\fill}}cccccccccc}
        \toprule
          $\overline{B}$ (minute) & $K$ (\$) & $M_1$ & $M_2$ & $R^\delta$ & $R^\nu$ & $\overline{T}$ (minute) & $W^\prime$ (\$)\\
        \midrule
          $120$ & $50000$ & $10^6$ & $3\times10^6$ & $0.568$ & $2.045$ & $86400$ & $1$\\
        \bottomrule
    \end{tabular*}}
\end{table}

\subsection{Computational performance of the heuristic methods}\label[sec]{computational_performance}
We conduct an analysis to reveal the computational performance of the Greedy and DaC methods, comparing them to the MILP solved by Gurobi using $|\sB|$ as a problem size determinant lever. We utilized data from a depot located in the Chicago metropolitan area. The parametric design outlined in the previous section served as the baseline. We randomly selected a subset of trips from the available trips of this depot with $|\sB|\in [10, 20, 30, 40, 50, 100, 200, 300]$ following a uniform distribution for the selection probability. A total of 2,230 instances were solved using the three methods, with a computational time limit of 1,200 seconds per instance.

In the Greedy method, it is assumed that all buses begin their daily trips with a fully charged battery. However, in the proposed MILP model, we allow the model to determine the required initial battery level dynamically. This assumption is made in the Greedy algorithm to simplify the model and enable it to handle large-scale problems more efficiently. To ensure a fair comparison between the Greedy, DaC methods, and the MILP solver, we change the proposed MILP model by incorporating an additional constraint. \cref{init_battery} ensures that the initial battery level of the buses is equal to their battery capacity. 

\begin{equation}\label[cons]{init_battery}
    v_{si} = \overline{B}y_{si} \qquad \forall i\in \sB 
\end{equation}

All computations were performed on a workstation equipped with an Intel{\textsuperscript \textregistered} Xeon{\textsuperscript \textregistered} Gold 6138 CPU @2.0 GHz, 128 GB of RAM, and 64 cores. The Python 3.8.8 interface to the commercial solver Gurobi 10.0 \citep{gurobi} was employed to solve the problem instances.

The computational performance of the MILP model solved with Gurobi, as well as the Greedy and DaC solution approaches, are reported in \cref{computational_exp}. The first column specifies the number of trips, while the second column indicates the number of instances solved for a given number of trips. For the MILP approach, the first column represents the number of instances where optimality was achieved. The second column for the MILP shows the average MIP gap percentage, which measures the difference between the objective value of the best-known feasible solution found and the best lower-bound found. The Greedy column displays the average percentage $\Delta$ gap, indicating the difference between the solution reported by the solver and the solution found by the Greedy method. The DaC columns provide information on the number of times each scenario is divided into subproblems, $m$, as well as the average percentage $\Delta$ gap between the solution reported by the solver and the solution found by the DaC method. A negative average percentage indicates that the solutions obtained by the Greedy or DaC methods were superior to those achieved by the MILP approach. Note that in these instances, MILP actually did not reach an optimal solution.

\begin{table}[!htb]
\caption{Computational performance of the heuristic methods}
\label{computational_exp}
\resizebox{\textwidth}{!}{%
\begin{tabular}{@{}cccccccccc@{}}
\toprule
\multirow{2}{*}{\# Trips} & \multirow{2}{*}{\# Solved} &  & \multicolumn{2}{c}{MILP}&  & Greedy&  & \multicolumn{2}{c}{DaC}\\ 
\cmidrule(lr){4-5} \cmidrule(lr){7-7} \cmidrule(l){9-10} 
&&& \# OPTS & Avg. MIP gap (\%) && Avg. $\Delta$ gap (\%) && $m$ & Avg. $\Delta$ gap (\%) \\ 
\midrule
10  & 1000 &  & 999 & 7.51E-06  &  & 11.69  &  & 2 & 9.13\\
20  & 500  &  & 497 & 4.69E-05  &  & 12.52  &  & 2 & 6.97\\
30  & 500  &  & 493 & 1.55E-04  &  & 13.35  &  & 4 & 17.76\\
40  & 100  &  & 98  & 1.97E-04  &  & 15.02  &  & 2 & 7.49\\
50  & 100  &  & 94  & 5.60E-04  &  & 16.20  &  & 4 & 15.81\\
100 & 10   &  & 8   & 0.035 &  & 15.39  &  & 2 & 8.00\\
200 & 10   &  & 0   & 15.5 &  & 3.65   &  & 4 & -1.01\\
300 & 10   &  & 0   & 37.6 &  & -21.74 &  & 4 & -25.87\\
\bottomrule
\end{tabular}%
}
\end{table}

In the analysis presented in \cref{gap}, it is evident that utilizing the MILP approach through a solver is only effective for handling small-scale problems. Comparatively, the DaC demonstrates slightly better performance compared to the Greedy approach with some exceptions reported on \cref{computational_exp}. Furthermore, the results shown in \cref{fig:gap_a} indicate that as the number of subproblems (i.e., the number of times a problem is divided) increases, the quality of the solutions decreases. However, despite the increase in the number of subproblems, there is an improvement in solution quality for larger cases (e.g., 100, 200, and 300 trips) as shown in \cref{fig:gap_b}. This is because the solution quality of the Greedy and DaC approaches is compared to sub-optimal solutions obtained from the MILP solver.

\begin{figure*}[!htbp]
    \centering
    \subfloat[\centering \label{fig:gap_a}]{%
        \includegraphics*[width=0.485\textwidth,height=\textheight,keepaspectratio]{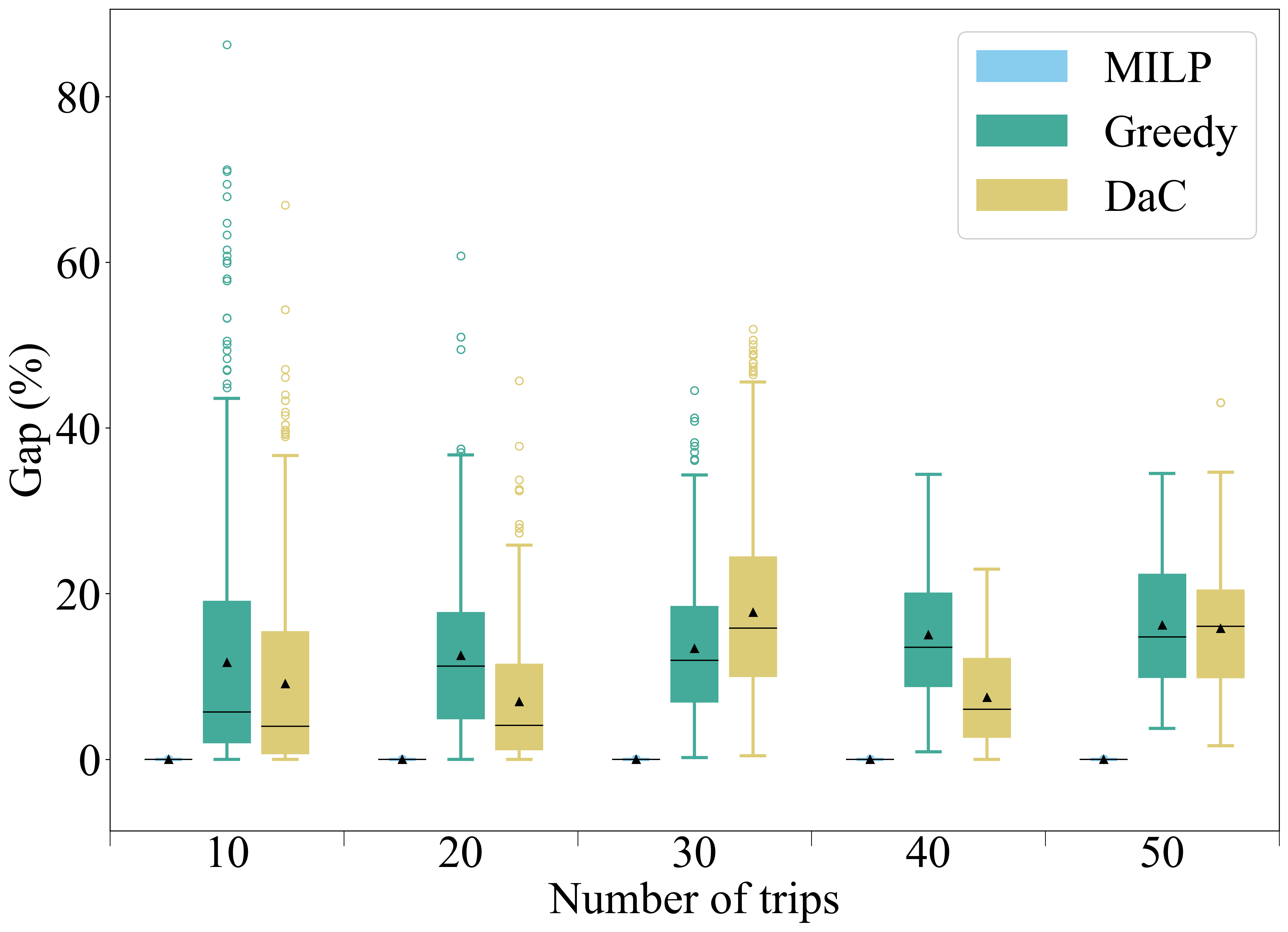}}
    \subfloat[\centering \label{fig:gap_b}]{\includegraphics*[width=0.495\textwidth,height=\textheight,keepaspectratio]{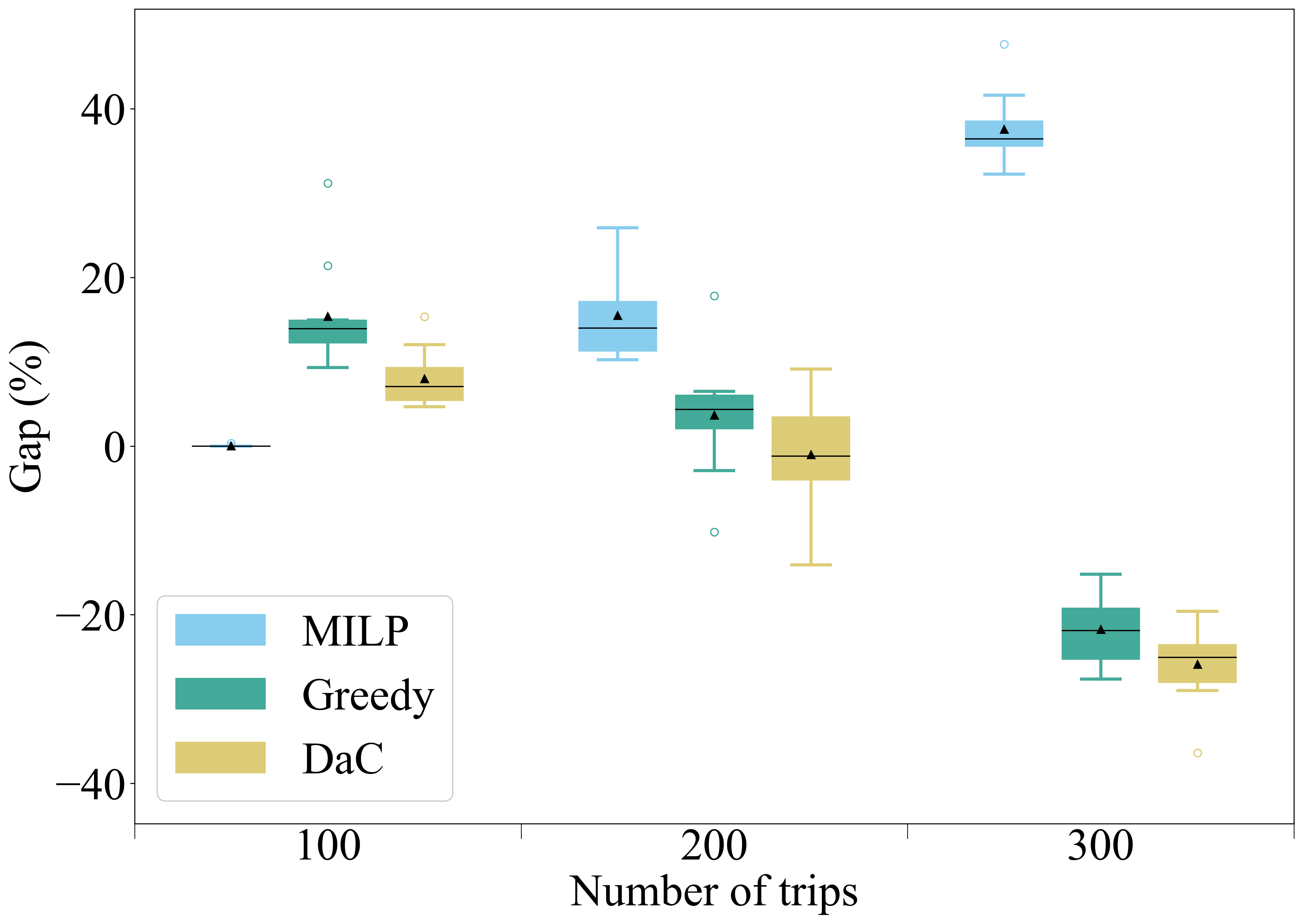}}
    \caption{Gap statistics for the three solution approaches.} \label[fig]{gap}
\end{figure*}

The solution times for the different approaches are visually represented in \cref{fig:time_res_a} and \ref{fig:time_res_b}. These figures clearly demonstrate that the Greedy method consistently outperforms the other approaches in terms of solution time. The DaC method also shows a faster performance compared to the MILP solver, although it is slightly slower than the Greedy method. However, it is important to note that as the problem size, measured by the number of trips, increases, the solution time for the DaC method experiences a substantial increase. 

Taking the problem size variability into account, each method exhibits its own strengths. The MILP solver performs well for small cases, where its optimal solutions can be effectively utilized. The DaC method proves to be effective for medium-sized cases, offering a balance between solution quality and computational efficiency. The Greedy method, on the other hand, excels in handling large cases by providing rapid solutions that are reasonably close to the solutions obtained by the MILP solver, for solution quality refer to \cref{computational_exp} and \cref{gap}. This demonstrates the efficiency of the Greedy method in terms of both speed and solution quality. As the Greedy is the fastest approach for very large-scale instances and finds reasonable solutions, it is utilized in \cref{case_studies}.

\begin{figure*}[!htbp]
    \centering
    \subfloat[\centering \label{fig:time_res_a}]{%
        \includegraphics*[width=0.48\textwidth,height=\textheight,keepaspectratio]{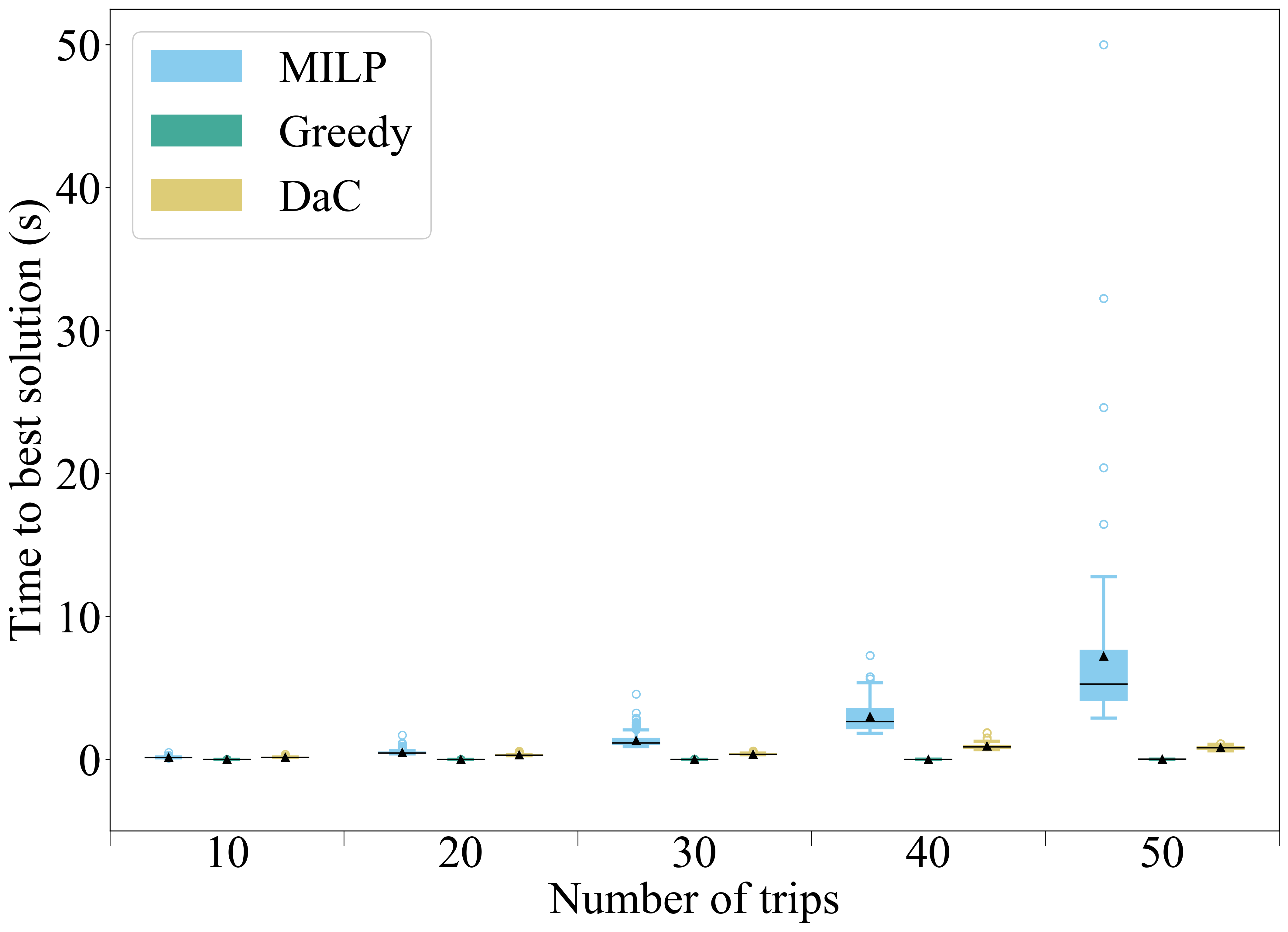}}
    \subfloat[\centering \label{fig:time_res_b}]{\includegraphics*[width=0.495\textwidth,height=\textheight,keepaspectratio]{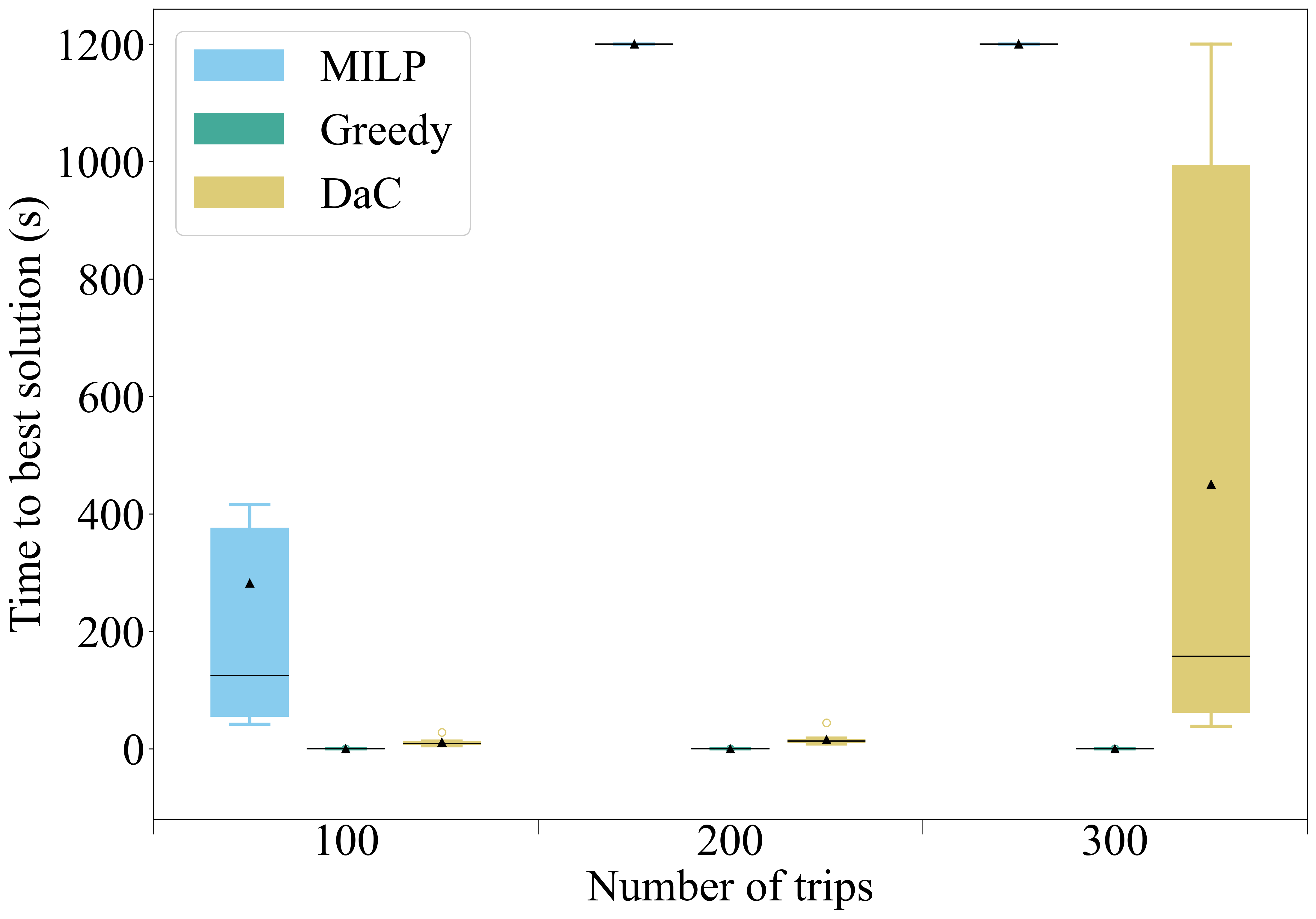}}
    \caption{Time-to-best-solution statistics for the three solution approaches.} \label[fig]{time_res}
\end{figure*}

\subsection{Large-scale case studies}\label[sec]{case_studies}
Case studies were conducted to provide insights for key metrics, such as share of EVs, number of EVs per each DV replaced, and share of revenue trip time over the day. These studies also demonstrate the applicability of the proposed approach at large-scale problem instances. We adopted the CTA, PACE, and CapMetro data as explained previously. The number of revenue trips for CTA, PACE, and CapMetro are nearly 18,700, 7,300, and 5,400, respectively. We consider three vehicle range lever: 60, 120, and 150 miles and three EV deployment target lever: Low, medium, and high. The deployment target is controlled through parameters $K$ and $W$, which are described in the \cref{methodology}. We also ran a \emph{DV only} scenario for each agency by solving the SDVSP allowing longer blocks followed by a version of the BCP without electrification constraints. Similarly, the leftover, longer than EV range, blocks in each electrification scenario are also chained into DVs using that version of the BCP.

\cref{percent_share} presents the percent share of EVs and DVs (i.e., the fleet decomposition) on the left y-axis, and the total number of buses on the right one. Since our method does not implement hard constraints on the block length, 100\% electrification is not guaranteed but as the results reveal, a near 100\% electrification is possible at the expense of a substantial fleet size increase. These are 54\%, 59\%, and 58\% for CTA, PACE, and CapMetro, respectively, in the case of high deployment and 150-mile range compared to DV only. Note that ``DV'' on x-axis of \cref{percent_share} and figures to be presented hereafter denotes the DV only scenario.

\begin{figure*}[!htb]
\centering
\subfloat[CTA.\label{fig:Slide1}]{%
\includegraphics*[width=0.48\textwidth,height=\textheight,keepaspectratio]{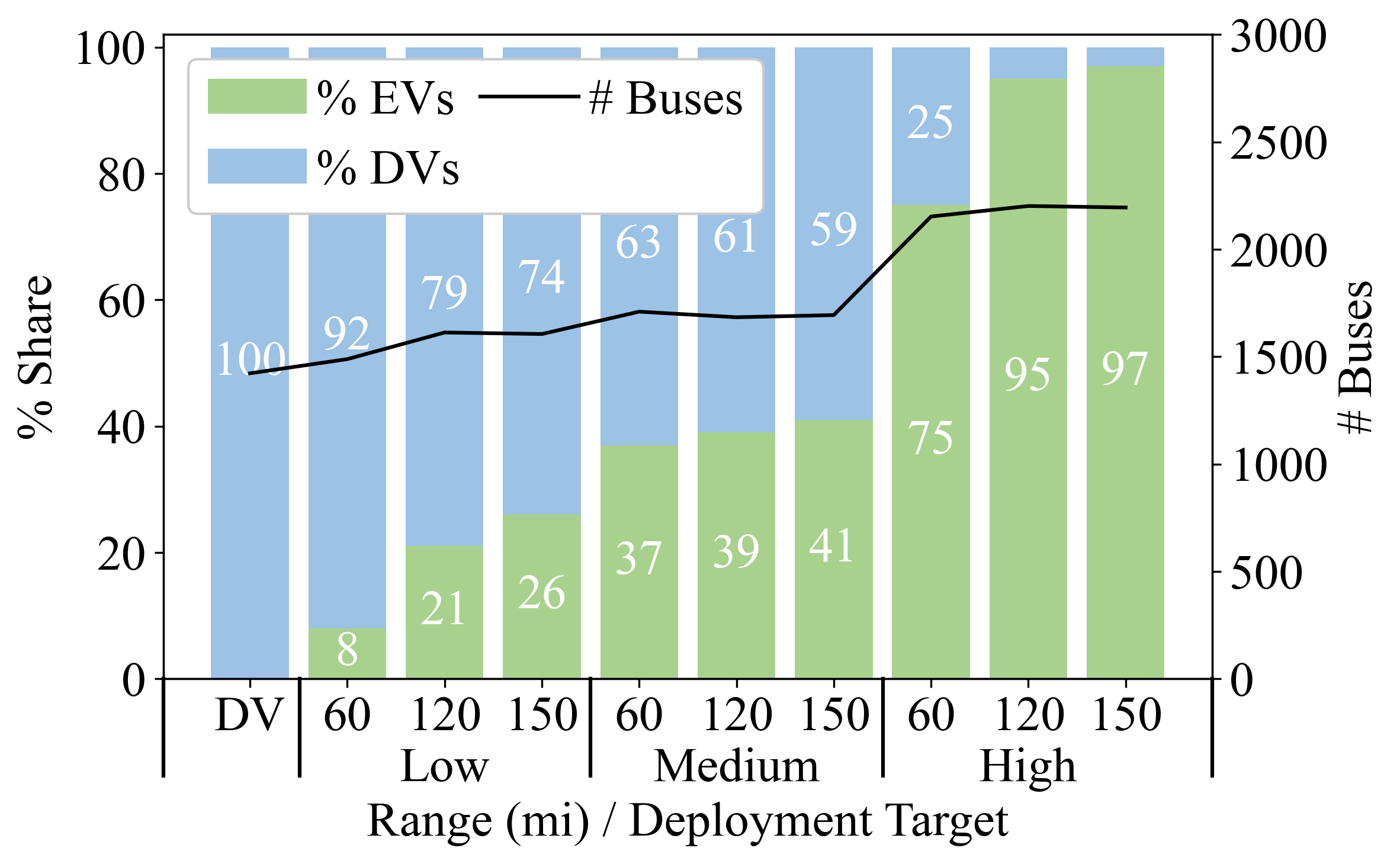}}\quad
\subfloat[PACE.\label{fig:Slide5}]{%
\includegraphics*[width=0.48\textwidth,height=\textheight,keepaspectratio]{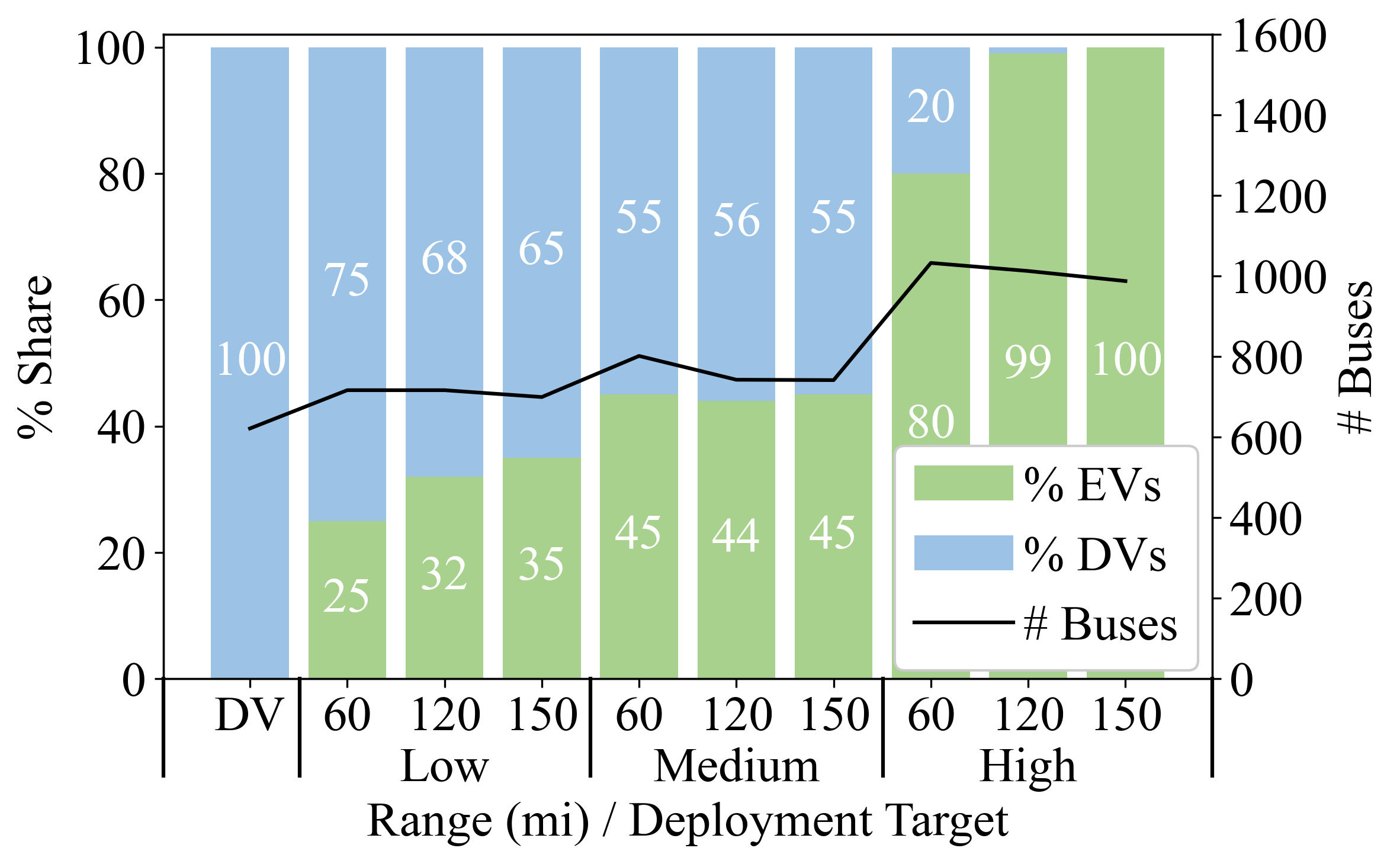}}

\subfloat[CapMetro.\label{fig:Slide9}]{%
\includegraphics*[width=0.48\textwidth,height=\textheight,keepaspectratio]{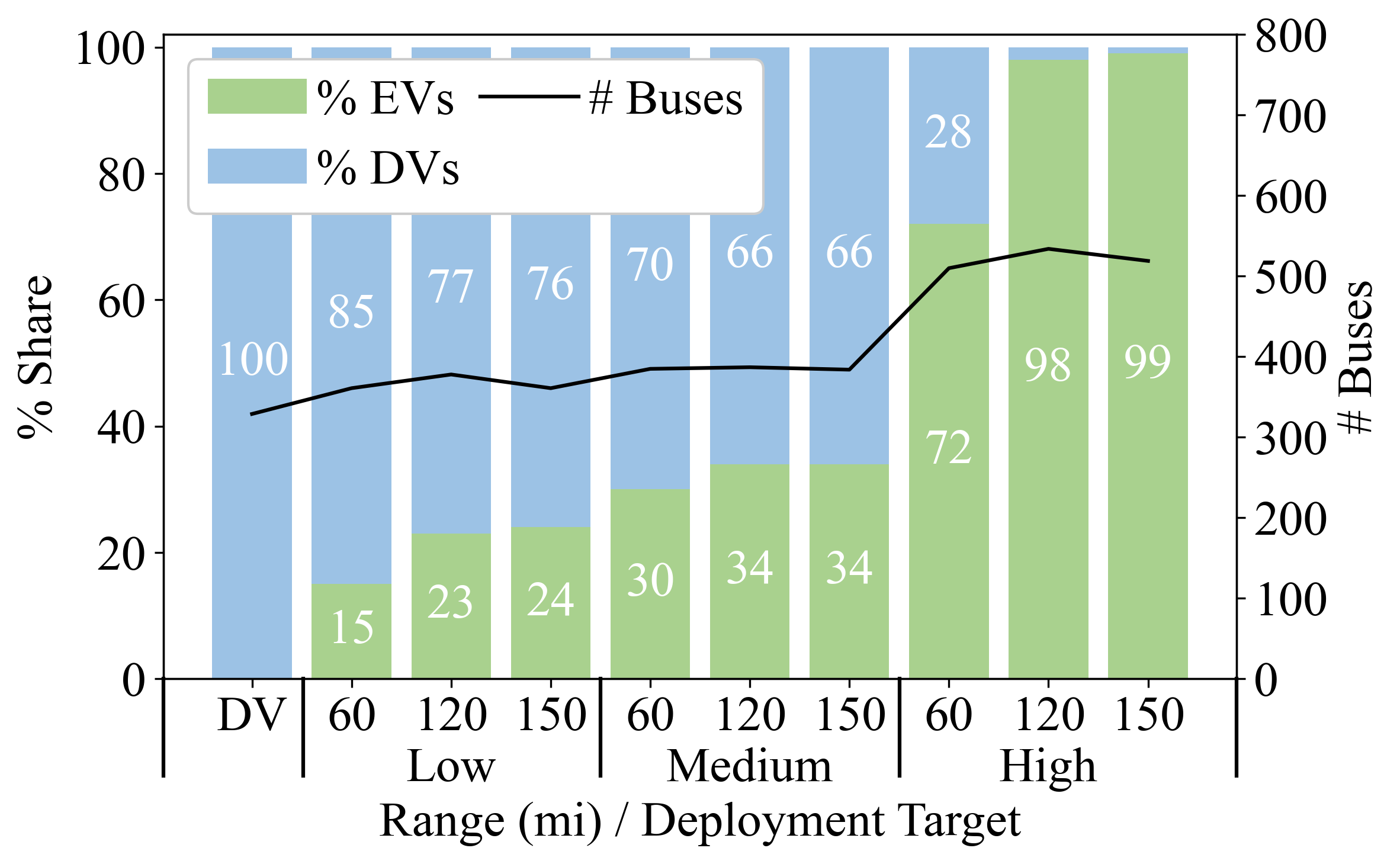}}
\caption{Percent share and number of buses.} \label[fig]{percent_share}
\end{figure*}

\cref{EV_DV_replacement} presents the number of EVs replacing one DV. This metric is of particular interest to transit agencies as it informs on an expected fleet size with EV deployment targets. Number of EVs in a given scenario is divided by the difference of DVs in the DV only and the given scenario to obtain this ratio. The ratio decreases as the EV range increases. This is intuitive because EVs become similar to DVs with increasing EV range. We do not observe such a strong relationship between deployment target and replacement ratio for a given EV range with the exception of 60-mile range, where there is a substantial decrease moving from low to medium.

\begin{figure*}[!htb]
\centering
\subfloat[CTA.\label{fig:Slide2}]{%
\includegraphics*[width=0.48\textwidth,height=\textheight,keepaspectratio]{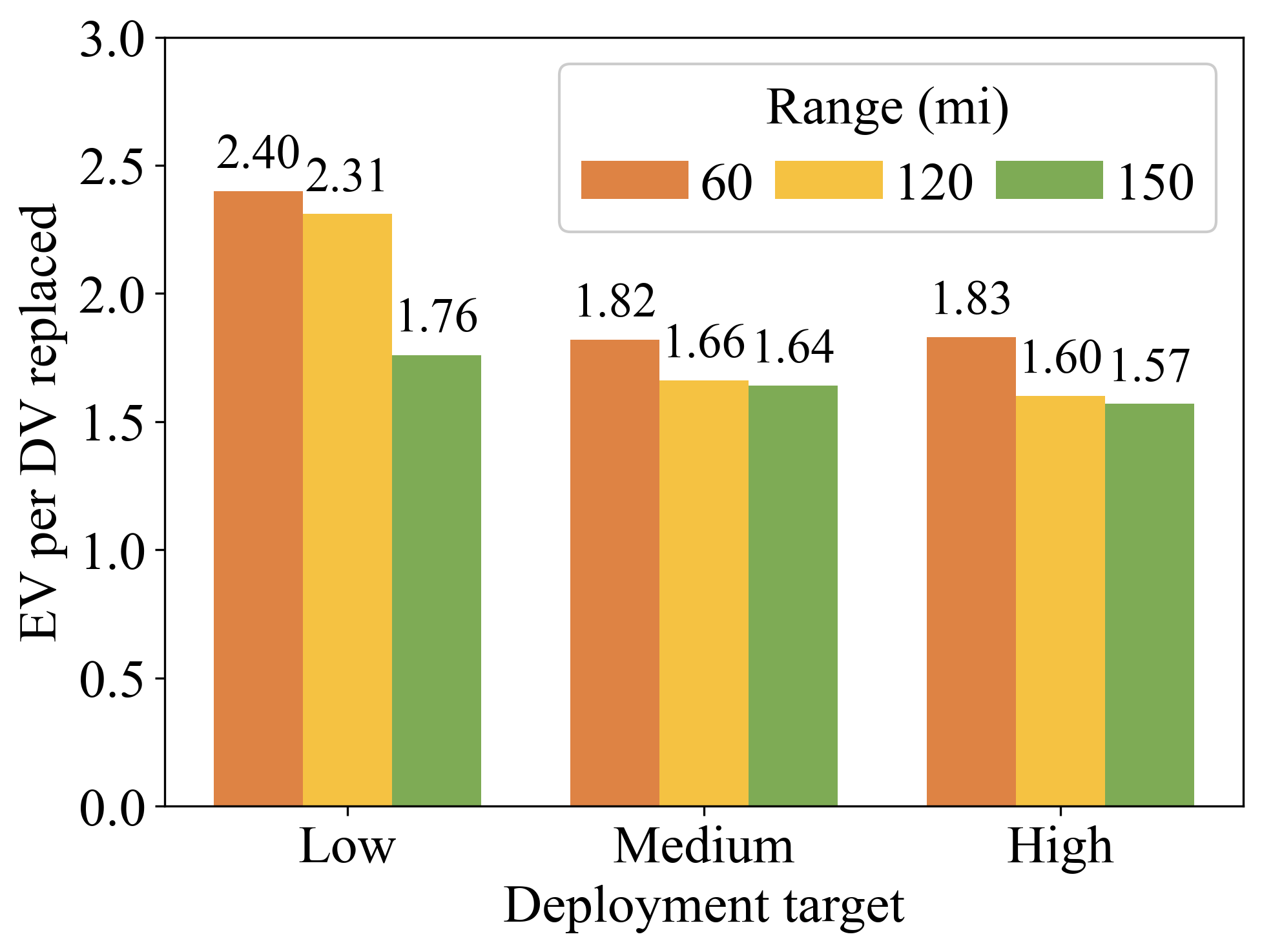}}\quad
\subfloat[PACE.\label{fig:Slide6}]{%
\includegraphics*[width=0.48\textwidth,height=\textheight,keepaspectratio]{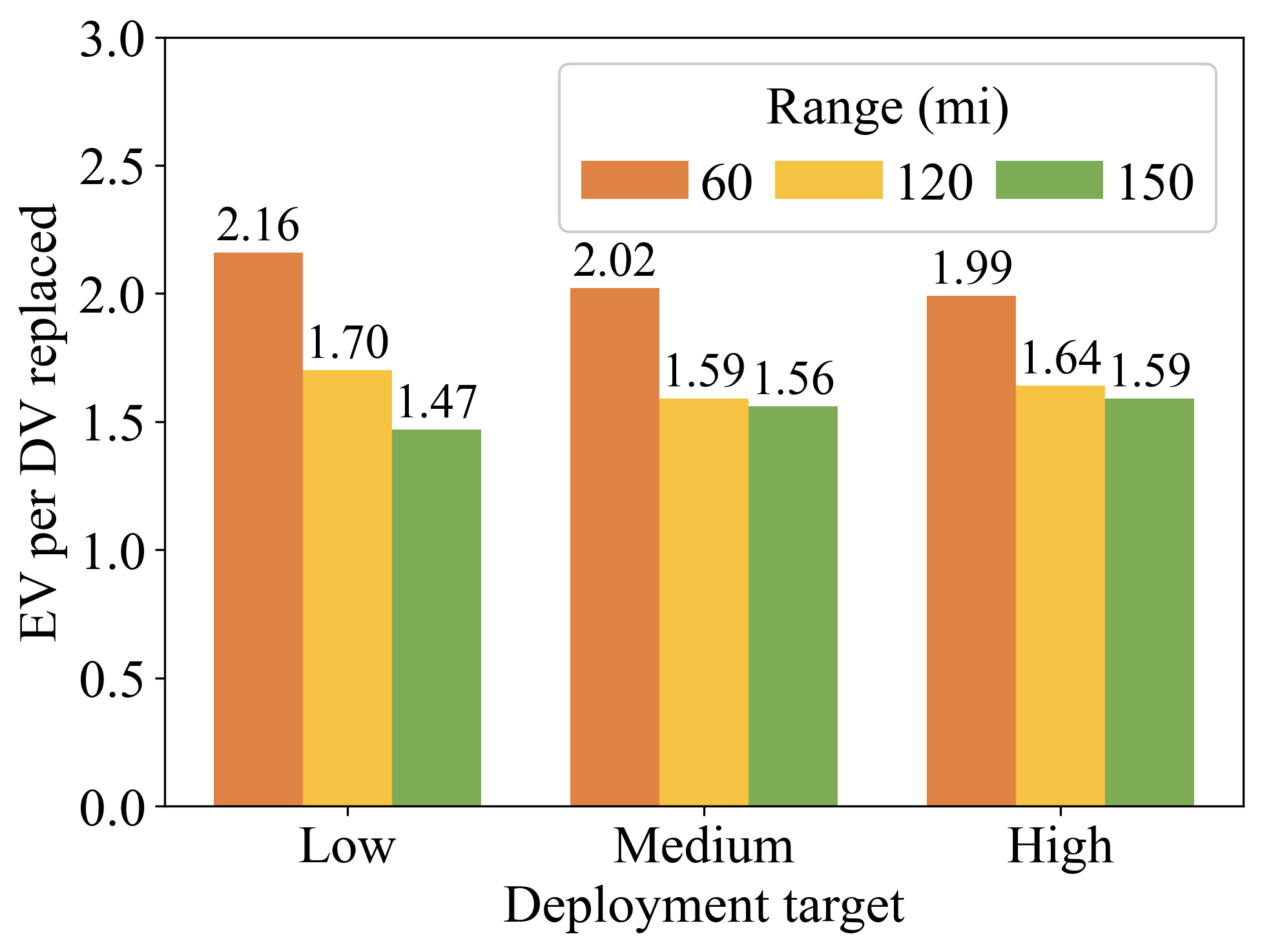}}

\subfloat[CapMetro.\label{fig:Slide10}]{%
\includegraphics*[width=0.48\textwidth,height=\textheight,keepaspectratio]{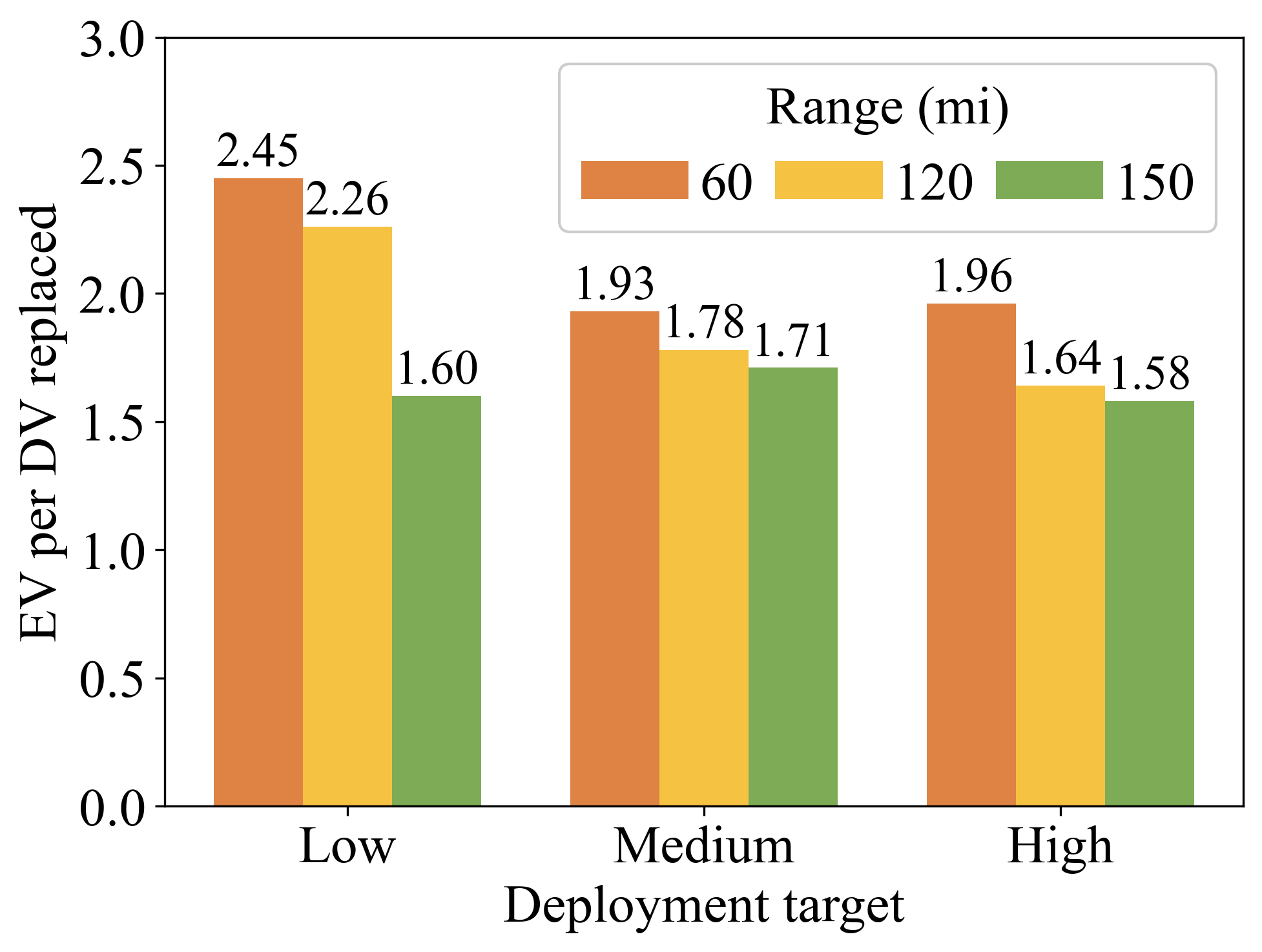}}
\caption{EV per DV replaced by deployment target.} \label[fig]{EV_DV_replacement}
\end{figure*}

The block efficiency is demonstrated in \cref{block_efficiency_stats} and is simply the ratio of revenue trip time to the entire block time. Comparing the high-deployment, 150-mile scenario to DV only, one observes an 18\% reduction in the share of revenue trip time for CTA and CapMetro, whereas a 20\% reduction is observed for PACE. Since blocks become shorter with higher EV deployment, there are more deadheading trips to and from the depot, which explains this change. Moreover, with higher EV deployment, there is more layover at the depots due to recharging, which is demonstrated in \cref{vehicle_efficiency_stats}. The vehicle schedule efficiency is calculated by dividing the revenue trip time to the entire horizon. In this case the efficiency decrease is by 35\% for CTA, and 37\% for PACE and CapMetro. Compared to the block efficiency, the drops are even more dramatic because there is also time loss due to recharging, and not only extra deadheading.

\begin{figure*}[!htb]
\centering
\subfloat[CTA.\label{fig:Slide3}]{%
\includegraphics*[width=0.48\textwidth,height=\textheight,keepaspectratio]{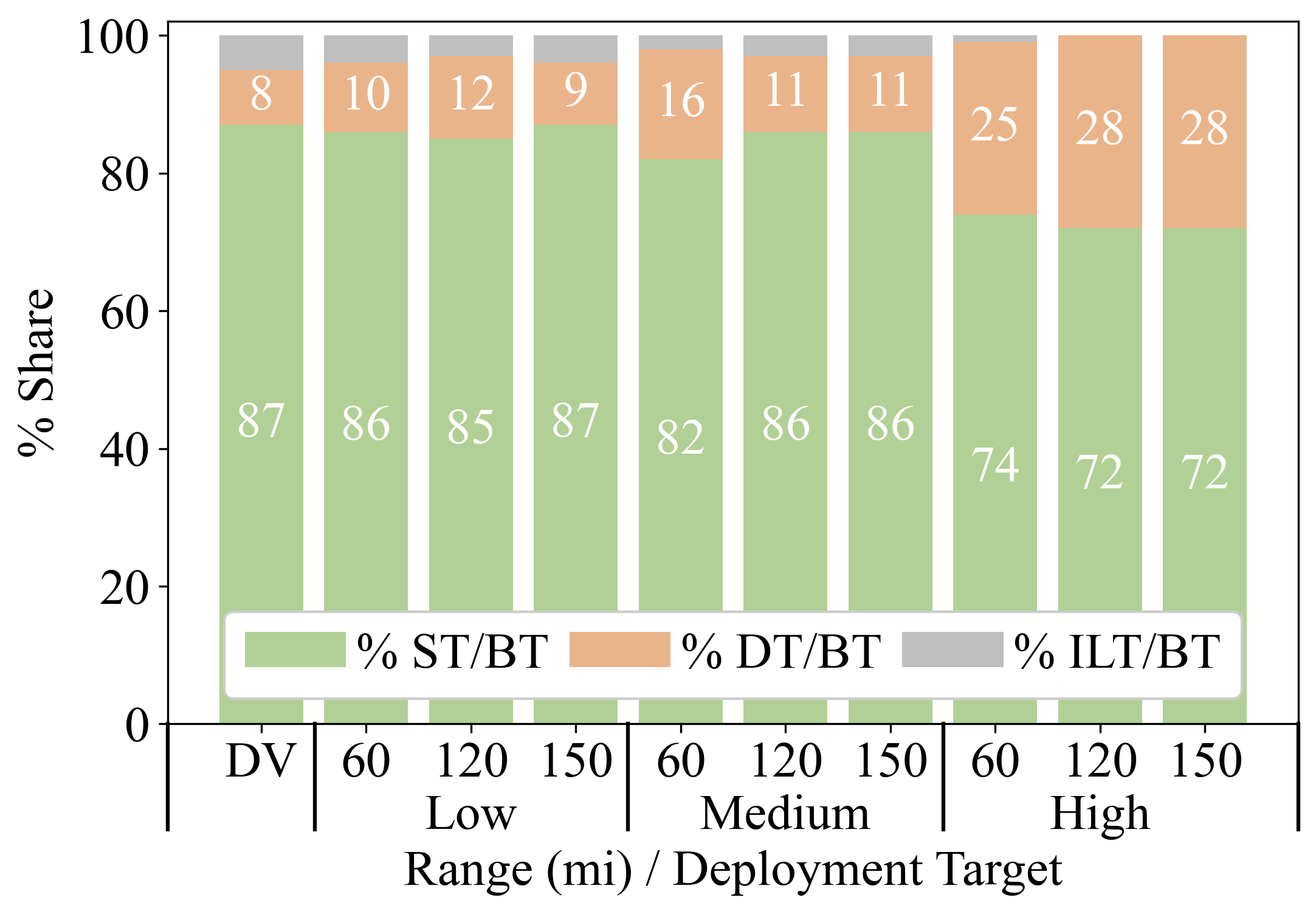}}
\quad\subfloat[PACE.\label{fig:Slide7}]{%
\includegraphics*[width=0.48\textwidth,height=\textheight,keepaspectratio]{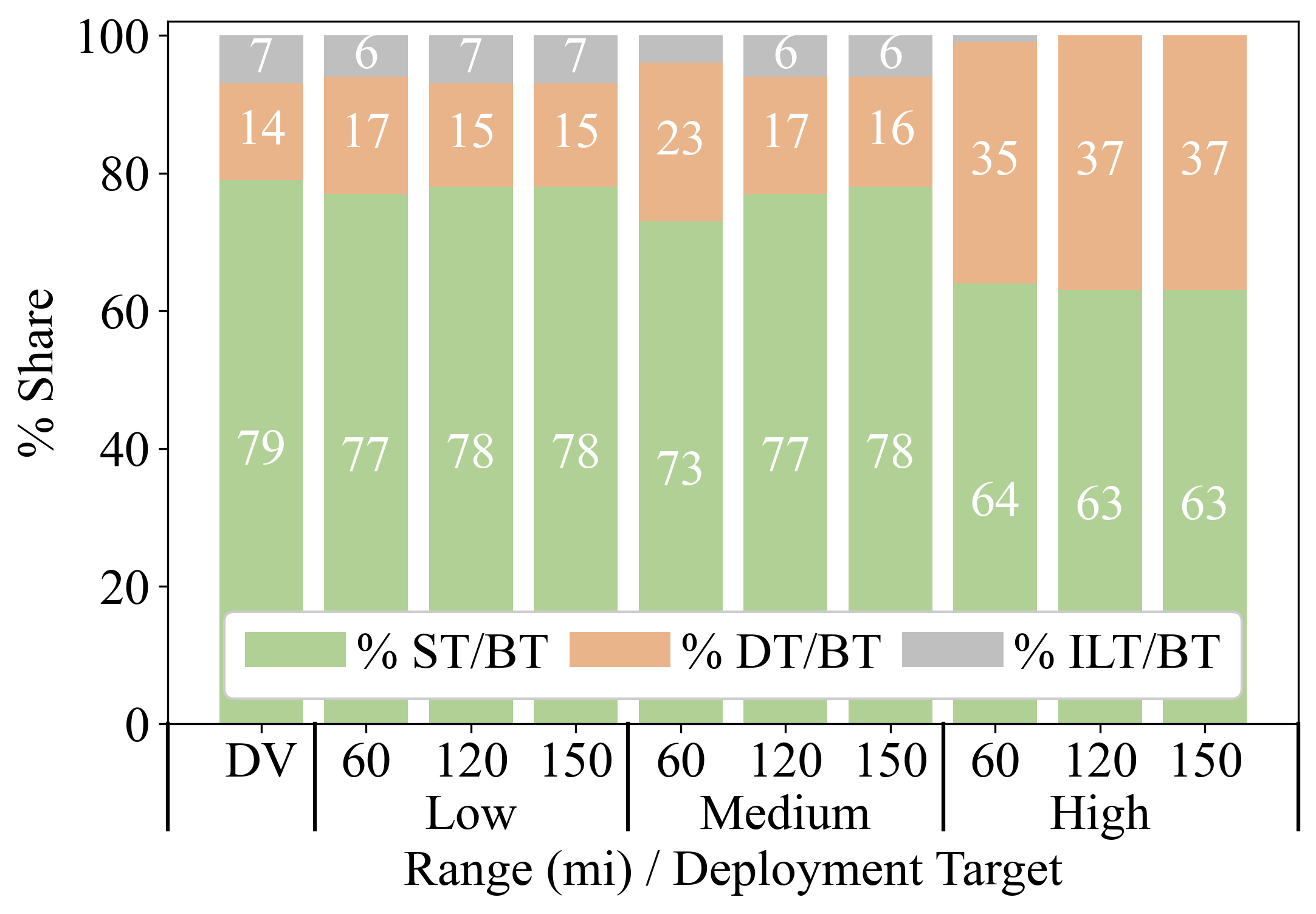}}

\subfloat[CapMetro.\label{fig:Slide11}]{%
\includegraphics*[width=0.48\textwidth,height=\textheight,keepaspectratio]{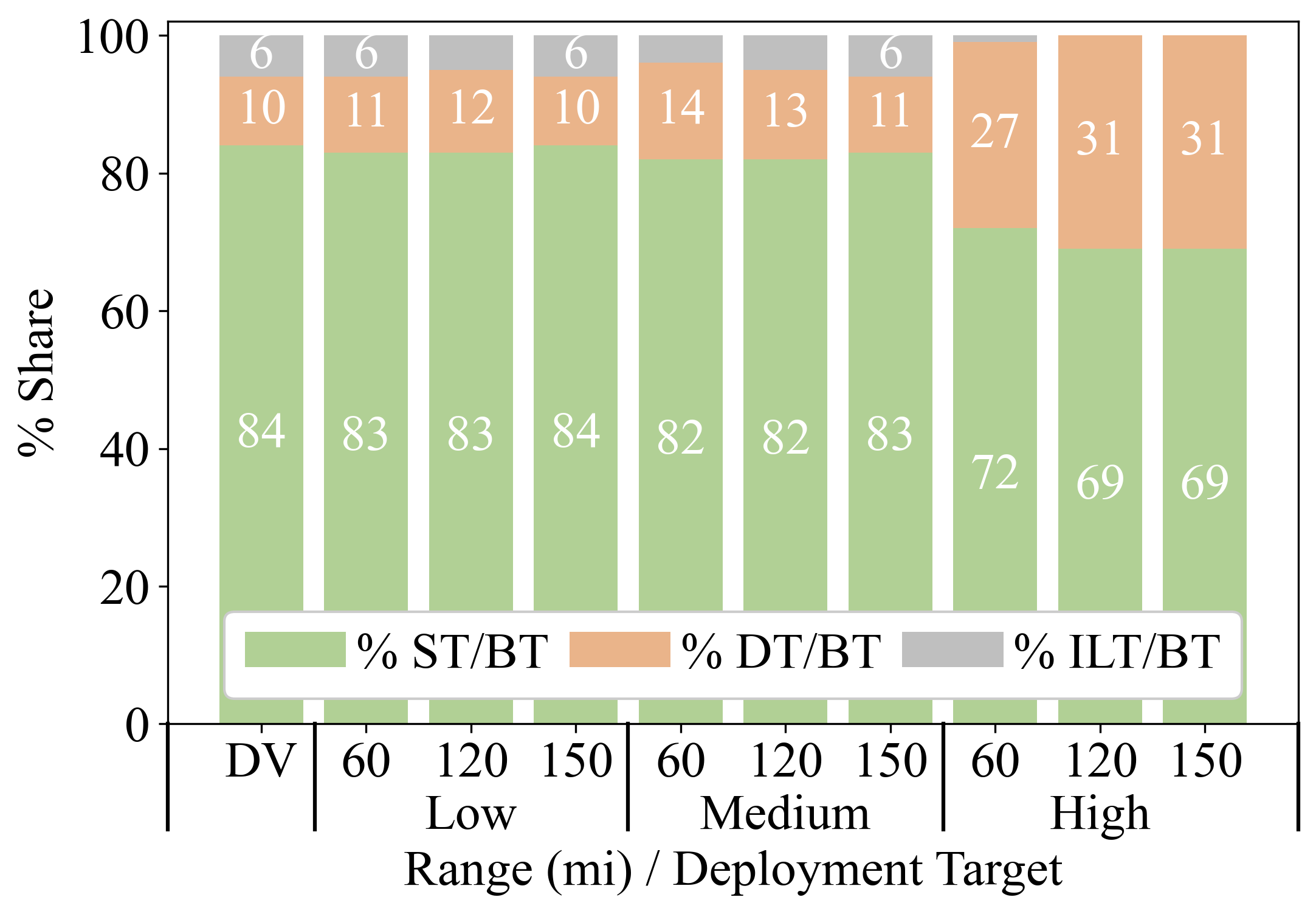}}
\caption{Block efficiency statistics. BT: Block time, ST: Service time, DT: Deadhead time, ILT: Intertrip layover time.} \label[fig]{block_efficiency_stats}
\end{figure*}

\begin{figure*}[!htb]
\centering
\subfloat[CTA.\label{fig:Slide4}]{%
\includegraphics*[width=0.48\textwidth,height=\textheight,keepaspectratio]{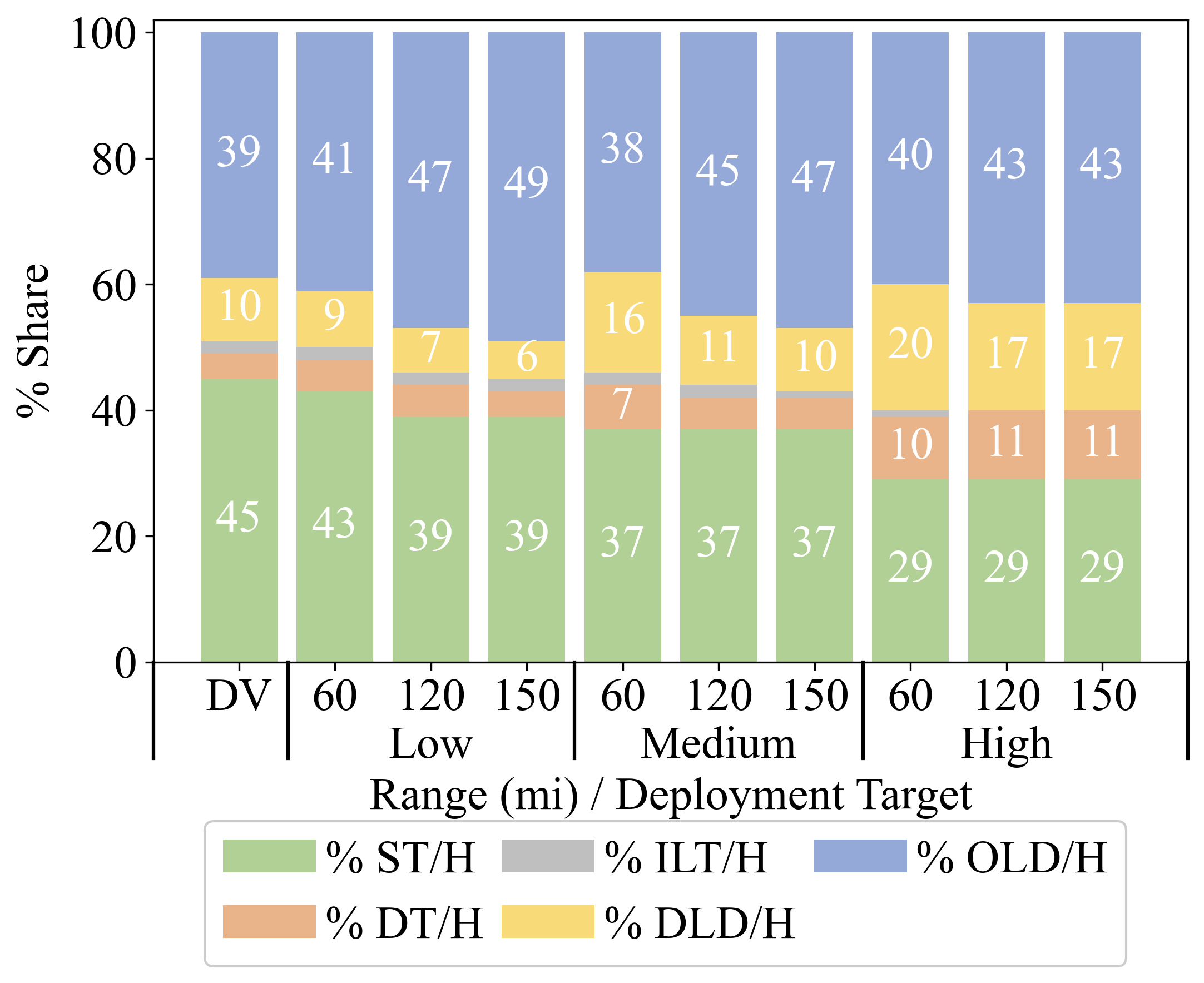}}\quad
\subfloat[PACE.\label{fig:Slide8}]{%
\includegraphics*[width=0.48\textwidth,height=\textheight,keepaspectratio]{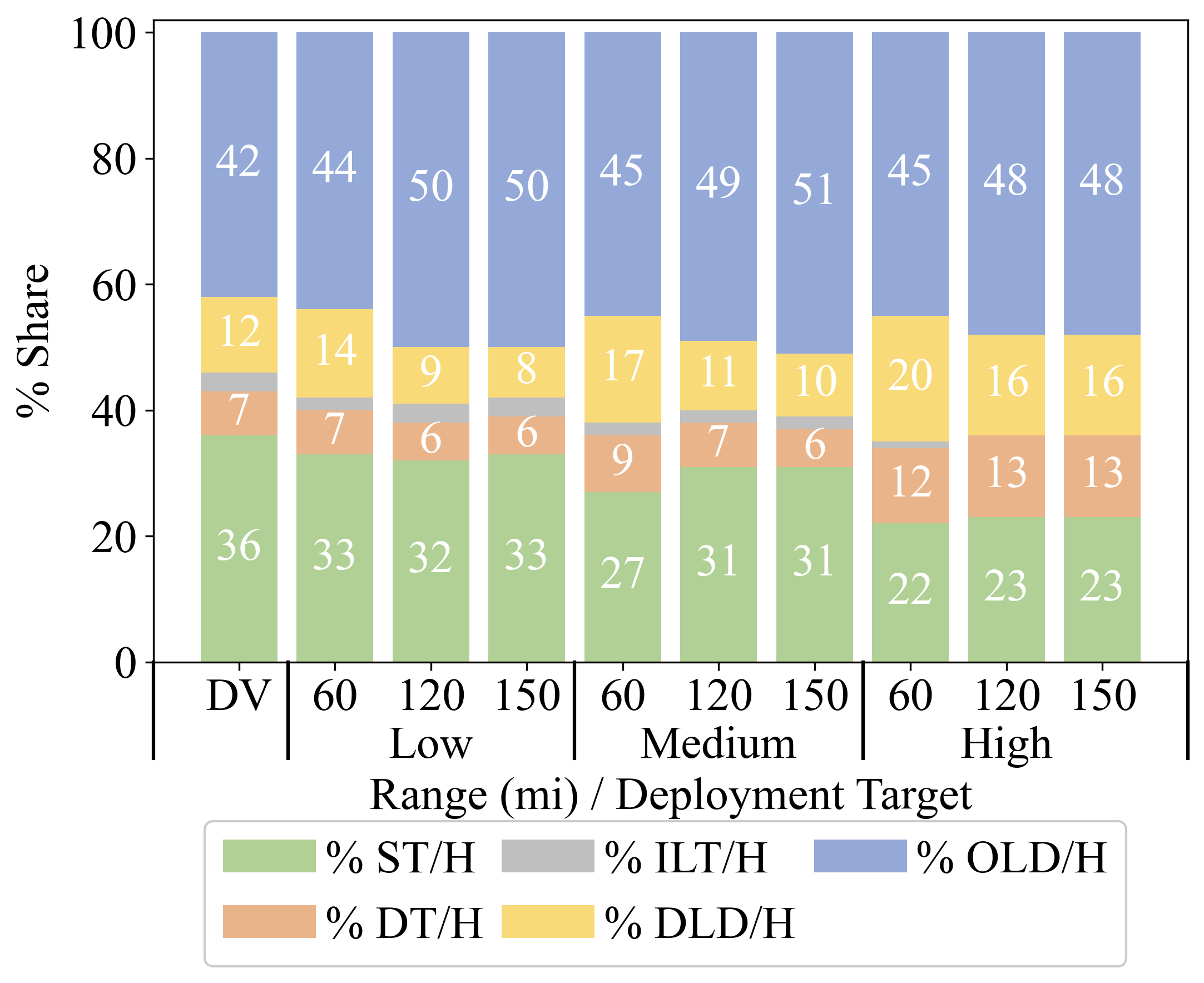}}

\subfloat[CapMetro.\label{fig:Slide12}]{%
\includegraphics*[width=0.48\textwidth,height=\textheight,keepaspectratio]{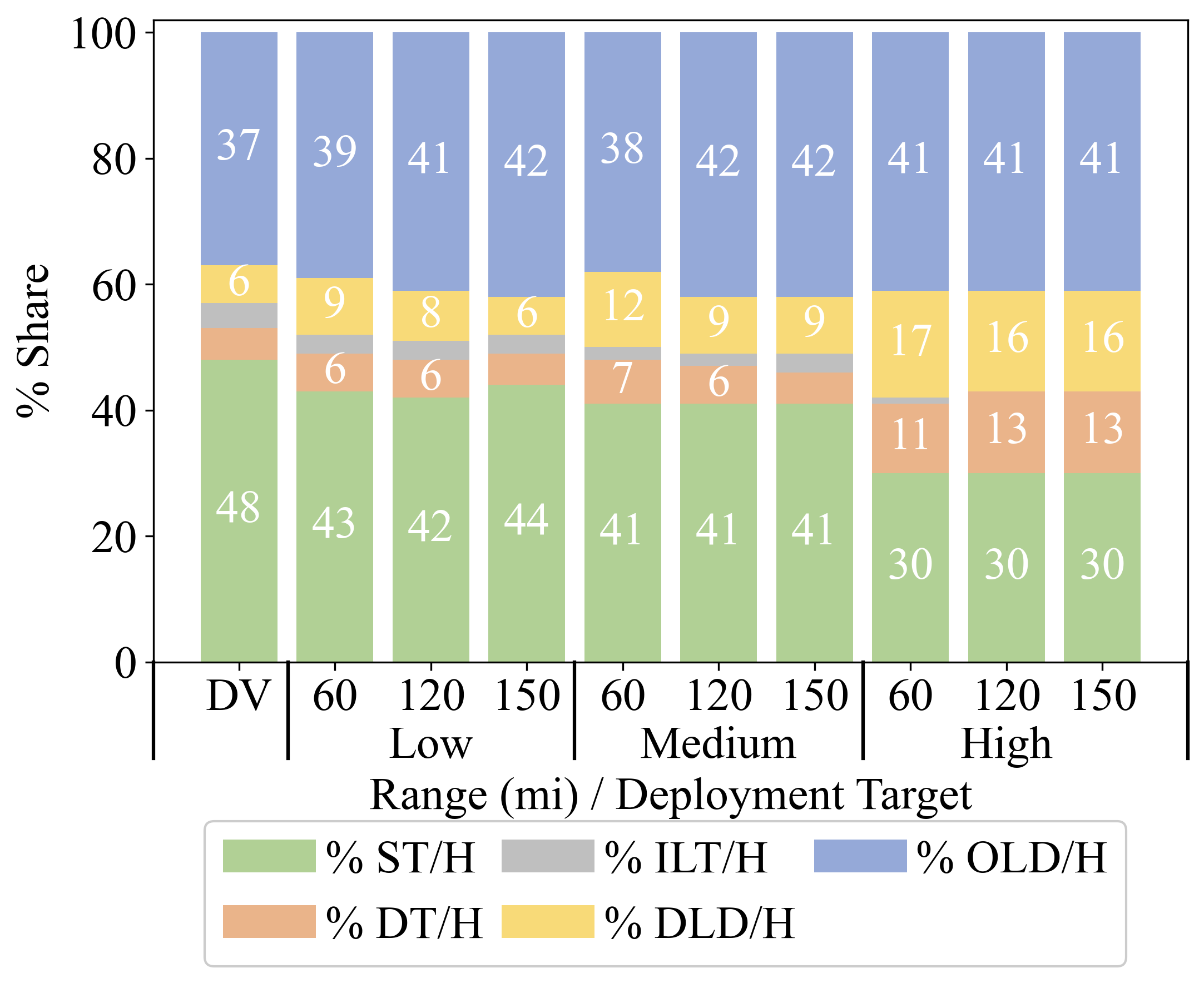}}
\caption{Vehicle schedule efficiency statistics. H: Horizon, ST: Service time, DT: Deadhead time, ILT: Intertrip layover time. DLD: Daytime layover at depot, OLD: Overnight layover at depot.} \label[fig]{vehicle_efficiency_stats}
\end{figure*}

\section{Conclusion}\label[sec]{conclusion}
In this study, we proposed a two-stage solution framework to solve the SDEVSP. We solve the SDVSP to generate blocks in the first stage and then solve the BCP to form vehicle schedules. While we utilized traditional solution methods to solve the SDVSP, three solution approaches, namely MILP, DaC, and Greedy, were developed. An extensive computational experiments conducted to compare these methods revealed solution quality and computational time trade-off. We observed that the Greedy method can solve large-scale instances considerably fast, and its solution quality is comparable to that of the MILP within reasonable solution time limits. 

Utilizing the greedy method, we conducted case studies for three transit agencies: CTA, PACE, and CapMetro. Near 100\% electrification is possible with a replacement ratio of $\sim$1.6 EVs per DV and a 150-mile range. However, vehicle schedule efficiency would decrease by $\sim$35\%. These results can be considered optimistic given our assumptions on depot size and charger availability. On the other hand, we do not consider opportunistic charging at the terminal stops, which would increase the schedule efficiency.

The SDEVSP is quite complex, and there are yet more aspects that are not considered in this study. Some of these are i) charger availability, ii) charger choice, e.g., pantograph or traditional, iii) charger level selection, e.g., 150 kW and 450 kW, iv) charger location including en-route charging, v) non-linear charge and discharge profiles, and vi) vehicle sizes, e.g., 40 ft and 60 ft. The readily difficult problem can easily become intractable considering a combination of these aspects. Therefore, we may tackle these problems in stages. The proposed two-step solution approach only finds a solution to the dauntingly challenging problem, and it can be enhanced. The greedy solution method is flexible to incorporate charger availability, charger level selection, and non-linear charge profiles. Future studies will improve the greedy algorithm and propose methods to address a subset of these aspects.

\section*{Acknowledgments}
This material is based upon work supported by the U.S. Department of Energy, Office of Science, under contract number DE-AC02-06CH11357. 
This report and the work described were sponsored by the U.S. Department of Energy (DOE) Vehicle Technologies Office (VTO) under the Systems and Modeling for Accelerated Research in Transportation (SMART) Mobility Laboratory Consortium, an initiative of the Energy Efficient Mobility Systems (EEMS) Program. 
Erin Boyd, a DOE Office of Energy Efficiency and Renewable Energy (EERE) manager, played an important role in establishing the project concept, advancing implementation, and providing guidance. The authors remain responsible for all findings and opinions presented in the paper. The findings are not suggestions for agencies to implement given the assumptions made in this study.

\clearpage
\bibliographystyle{elsarticle-harv} 
\bibliography{our_bib}

\vfill
\framebox{\parbox{.90\linewidth}{\scriptsize The submitted manuscript has been created by
        UChicago Argonne, LLC, Operator of Argonne National Laboratory (``Argonne'').
        Argonne, a U.S.\ Department of Energy Office of Science laboratory, is operated
        under Contract No.\ DE-AC02-06CH11357.  The U.S.\ Government retains for itself,
        and others acting on its behalf, a paid-up nonexclusive, irrevocable worldwide
        license in said article to reproduce, prepare derivative works, distribute
        copies to the public, and perform publicly and display publicly, by or on
        behalf of the Government.  The Department of Energy will provide public access
        to these results of federally sponsored research in accordance with the DOE
        Public Access Plan \url{http://energy.gov/downloads/doe-public-access-plan}.}}
\end{document}